\documentclass{article}
\usepackage[margin=1.0in]{geometry}
\usepackage[dvips]{graphicx}
\usepackage{color,amsmath,latexsym,amssymb,amsthm}
\usepackage{mathtools}
\usepackage{ifthen}
\usepackage{mathrsfs}
\usepackage{hyperref}
\usepackage{cleveref}
\usepackage{enumitem}
\usepackage{subfig}

\newtheorem{definition}{Definition}
\numberwithin{definition}{section}
\numberwithin{equation}{section}
\newtheorem{theorem}[definition]{Theorem}
\newtheorem{lemma}[definition]{Lemma}
\newtheorem{corollary}[definition]{Corollary}
\newtheorem{proposition}[definition]{Proposition}
\theoremstyle{definition}

\newtheorem{assumption}{Assumption}

\crefdefaultlabelformat{#2\textup{#1}#3}

\Crefname{assumption}{Assumption}{Assumptions}
\crefname{assumption}{Assumption}{Assumptions}

\newlist{enumthm}{enumerate}{1}
\setlist[enumthm]{label=\textup{(\alph*)},ref=\theassumption~\textup{(\alph*)}}
\crefalias{enumthmi}{assumption}

\newcommand{\funArg}[1]{\ifthenelse{\equal{#1}{}}  %if #1={} then print nothing, else (#1)
	{}
	{[#1]}
}
\newcommand{\FunArg}[1]{\ifthenelse{\equal{#1}{}}  %if #1={} then print nothing, else (#1)
	{}
	{({#1})}
}
\newcommand{\Du}{\mathrm{D}}
\newcommand{\vadcsp}{V_{\text{ad}}}
\newcommand{\pobj}{J}
\newcommand{\rpobj}{f}
\newcommand{\embedding}{\hookrightarrow}
\newcommand{\erpobj}{F}
\newcommand{\csp}{U}

\newcommand{\adcsp}{\mathcal{D}(\psi)}
\newcommand{\adsp}{W}
\newcommand{\inner}[3][]{( #2, #3 )_{#1}}
\newcommand{\norm}[2][2]{\|#2\|_{#1}}
\newcommand{\Caratheodory}{Carath\'eodory}

\newcommand{\dualp}[3][]{\langle #2, #3 \rangle_{{#1}^*, #1}}
\newcommand{\dualpHzeroone}[3][]
{\langle #2, #3 \rangle_{H^{-1}(#1), H_0^1(#1)}}
\newcommand{\wto}{\rightharpoonup}
\newcommand{\domain}{D}
\newcommand{\lb}{\mathfrak{l}}
\newcommand{\ub}{\mathfrak{u}}
\DeclarePairedDelimiterXPP\E[1]{\mathbb{E}}[]{}{#1}
\newcommand{\objective}{G}

\renewcommand{\abstract}[1]
{
	{\small
		\textbf{Abstract.} {#1}
		\\
	}
}

\newcommand{\keywords}[1]
{
	{\small
		\textbf{Key words.} {#1}
		\\
	}
}

\newcommand{\amssubject}[1]
{
	{\small
		\textbf{AMS subject classifications.} {#1}
	}
}

\title{Empirical risk minimization for risk-neutral composite optimal control
with applications to bang-bang control\footnote{\textbf{Funding:} This material is based upon work supported by the National Science Foundation
under Grant No.\ DMS-2410944.}}

\author{Johannes Milz\thanks{H.\ Milton Stewart School of Industrial and Systems Engineering, Georgia Institute of Technology, Atlanta, Georgia 30332 (johannes.milz@isye.gatech.edu).} \and Daniel Walter\thanks{Institut für Mathematik, Humboldt-Universität zu Berlin, 10117 Berlin (daniel.walter@hu-berlin.de).}}

\begin{document}

\maketitle

\abstract{%
    Nonsmooth composite optimization problems under uncertainty are prevalent in various scientific and engineering applications.
    We consider risk-neutral composite optimal control problems, where the objective function is the sum of
    a potentially nonconvex expectation function and a nonsmooth convex function.
    To approximate the risk-neutral optimization problems, we use a Monte Carlo sample-based approach,
    study its asymptotic consistency, and derive nonasymptotic sample size estimates.
    Our analyses leverage problem structure commonly encountered in PDE-constrained optimization problems, including compact embeddings and growth conditions. We apply our findings to bang-bang-type optimal control problems and propose the use of a conditional gradient method to solve them effectively.
    We present numerical illustrations.
 }

\par
\par
\amssubject{%
	90C15, 90C59, 90C06, 35Q93, 49M41, 90C60
}

\section{Introduction}

We consider composite risk-neutral optimization problems of the form
\begin{equation*}
\label{eq:nonconvexriskneutral}
\min_{u \in \csp} \, \objective(u)\coloneqq \erpobj(Bu) + \psi(u),
\quad \text{where} \quad \erpobj(w) \coloneqq \E{\rpobj(w,\boldsymbol{\xi})},
\tag{P}
\end{equation*}
and $B \colon  \csp \to W$ is a compact linear operator
between the Hilbert space $\csp$ and
the Banach space $W$. Moreover, $f \colon W \times \Xi \to [0,\infty)$
is a sufficiently smooth integrand, and $\psi : \csp \to [0,\infty]$ is a proper, closed, and convex but potentially nonsmooth function. Moreover, $\boldsymbol{\xi}$ is a random element taking values in a complete, separable metric
space $\Xi$. While~\eqref{eq:nonconvexriskneutral} covers a variety of challenging settings, the focus of the present work lies on risk-neutral partial differential equation (PDE)-constrained optimal control problems
\begin{align} \label{eq:intro:nonconvexPDE}
	\min_{u\in L^2(\domain)}\,
	\E{\pobj(S(Bu,\boldsymbol{\xi}))} + \psi(u),
\end{align}
where $J$ is a tracking-type functional,
and the choice of the deterministic control~$u$ influences the behavior of the solution $S(Bu,\boldsymbol{\xi})$ to a PDE with random inputs $\boldsymbol{\xi}$.

Nonconvex optimization problems governed by differential equations
arise in a multitude of application areas, such as
sensor placement \cite{An2022}, resource assessment of renewable tidal energy \cite{Funke2016},
and design of groundwater remediation systems \cite{Guan1999}. It is well known that the proper choice of the penalty function~$\psi$ in Problem~\eqref{eq:intro:nonconvexPDE} promotes certain structural features in its minimizers~$u^*$. For example, setting
\begin{align*}
    \psi(u)=I_{U_{\text{ad}}}(u)+  \beta \|u\|_{L^1(D)},  \quad \text{with} \quad  U_{\text{ad}} \coloneqq \{\, u \in L^2(\domain) \;|\; -1 \leq u(x) \leq 1 \quad \text{for a.e.}~x\in \domain \, \},
\end{align*}
where
$I_{U_{\text{ad}}}$ is the   indicator function of $U_{\text{ad}}$, and~$\beta>0$, tends to provide solutions satisfying a \textit{bang-bang-off} principle, that is, $u^*(x) \in \{-1,0,1\}$ for a.e.\ $x\in\domain$.
Hence, composite optimal control problems also arise from relaxations of mixed-integer
optimal control problems \cite{Funke2016}.

To handle the expected value in~\eqref{eq:nonconvexriskneutral} numerically, we approximate it using the
sample average approximation (SAA) approach \cite{Kleywegt2002}.
More specifically,
given a sequence $\boldsymbol{\xi}^1, \boldsymbol{\xi}^2, \ldots, \boldsymbol{\xi}^N, \ldots$ defined on a complete probability space
$(\Omega, \mathcal{F}, P)$
of independent $\Xi$-valued random elements
such that each $\boldsymbol{\xi}^i$ has the same distribution as $\boldsymbol{\xi}$, we consider the SAA problem,
also known as the empirical risk minimization problem,
\begin{equation*}
\label{eq:nonconvexsaa}
\min_{u \in \csp} \, \hat{\erpobj}_N(Bu,\omega) + \psi(u),
\quad \text{where} \quad \hat{\erpobj}_N(w, \omega) \coloneqq  \frac{1}{N}
\sum_{i=1}^{N}\rpobj(w,\boldsymbol{\xi}^i(\omega)).
\tag{P\textsubscript{$N$}}
\end{equation*}
For brevity, we often omit the second argument
of $\hat{\erpobj}_N$.

A central topic  in this context is whether this type of approximation is asymptotically consistent and whether the proximity of its solutions and critical points to their counterparts in~\eqref{eq:nonconvexriskneutral} can be quantified in terms of the sample size~$N$.
Although this already represents a formidable problem for sufficiently smooth objective functions, the potential lack of strong convexity of~$\psi$ aggravates the problem further, adding additional challenges to the theoretical analysis of the problem and its efficient solution.
\subsection{Contributions} \label{subsec:contribution}
In the present paper, we investigate the SAA approach for nonsmooth, potentially nonconvex risk-neutral minimization problems of the form~\eqref{eq:nonconvexriskneutral} from a qualitative and a quantitative perspective. Since Problem~\eqref{eq:nonconvexriskneutral} is potentially nonconvex, our analysis addresses the behavior of its solutions and critical points. The latter refers to points~$u^*_N \in \csp$ satisfying the first-order necessary subdifferential inclusion
\begin{align*}
    -B^* \Du \hat{F}_N (Bu^*_N) \in \partial  \psi(u^*_N)
\end{align*}
or, equivalently, which are zeros of the \textit{SAA gap functional}
\begin{align}
\label{eq:SAAgapfunctional}
\hat{\Psi}_N (u) &\coloneqq \sup_{v \in \adcsp}\,
\lbrack \, \inner[\csp]{ B^* \Du \hat{F}_N (Bu)}{u-v} +\psi(u)-\psi(v)\, \rbrack,
\end{align}
where $\adcsp$ is the domain of $\psi$.
Critical points as well as the associated gap function~$\Psi$ for Problem~\eqref{eq:nonconvexriskneutral} are defined analogously.
More specifically, the contributions of the paper are as follows:
\begin{enumthm}
    \item For sufficiently smooth functions $F$, we show that the optimal values,  solutions, and critical points of Problem~\eqref{eq:nonconvexsaa} converge asymptotically to their deterministic counterparts of Problem~\eqref{eq:nonconvexriskneutral} with probability one (abbreviated as ``w.p.~$1$'') as the sample size~$N$ tends to infinity.
    \item We establish nonasymptotic expectation bounds for the gap functional 	evaluated at sequences~$(u_N) \subset U$ of random vectors. More precisely, we show that, for every~$\varepsilon>0$,
    \begin{align*}
        \E{\Psi(u_N)} \leq \E{\hat{\Psi}_N(u_N)}  + \varepsilon \quad \text{if} \quad N \geq \frac{c_1}{\varepsilon^2} \cdot
        \ln\big(\mathcal{N}(\varepsilon/c_1; B(\adcsp))\big),
    \end{align*}
     where $\mathcal{N}(\nu; B(\adcsp))$ with $\nu > 0$
     denotes the $\nu$-covering number of the compact set $B(\adcsp)\subset W$, and $c_1 > 0$ is a constant that explicitly depends on the
     radius of the domain of $\psi$ as well as the integrand~$f$.
\item While Problem~\eqref{eq:nonconvexriskneutral} is typically not strongly convex, we provide nonasymptotic expectation bounds for the approximation of minimizers~$u^*$ to~\eqref{eq:nonconvexriskneutral} by  solutions to Problem~\eqref{eq:nonconvexsaa} if the integrand is convex for
each $\xi \in \Xi$,
and a growth-type condition on the partial linearization of~$\objective$ at a minimizer~$u^*$
holds true. Specifically, if
    	\begin{align}
	\label{eq:nonconvex:growth}
	\inner[\csp]{B^*\Du\erpobj(Bu^*)}{u-u^*}
	+ \psi(u) - \psi(u^*) \geq
	\theta \norm[\mathcal{U}]{u-u^*}^{2}
	\quad \text{for all} \quad u \in \csp
	\end{align}
for some $\theta >0$ and $U$ continuously embeds into the Banach space~$\mathcal{U}$, then
    \begin{align*}
	\E{\norm[\mathcal{U}]{u_N^*-u^*}^{2}}+ \E{\Psi(u^*_N)}+\E{\objective(u^*_N)- \objective(u^*)}
	\leq (1/N)c_2,
	\end{align*}
 where $c_2 > 0$ depends on
 the integrand, and $\theta$.
 Hence, the SAA solutions converge in expectation at the usual Monte Carlo rate  while both the value of the gap functional and the suboptimality in~$\objective$ exhibit superconvergence at a rate of~$(1/N)$.  This convergence statement is valid for linear bounded operators
 $B$ without  requiring their compactness.
\end{enumthm}

Our framework is applied to PDE-constrained optimization problems governed by affine-linear and bilinear elliptic equations, which allow for the use of bang-bang-off regularization terms. Reproducible numerical experiments empirically verify our theoretical results and further highlight the utility of the SAA approach for infinite-dimensional, nonsmooth problems.

\subsection{Related work}

Monte Carlo sample-based approximations are common
discretization approaches for risk-neutral huge-scale
optimization, particularly in PDE-constrained optimization \cite{Milz2022},  and
stochastic optimization \cite{Shapiro2021}. The theoretical analyses of this approximation approach
may be categorized into asymptotic and nonasymptotic ones.
In PDE-constrained
optimization under uncertainty with strongly convex control regularizers, the
papers \cite{Hoffhues2020,Martin2021,Milz2022c,Roemisch2021} provide
nonasymptotic analyses, such as sample size estimates and central limit-type
theorems. Moreover, \cite{Milz2022,Milz2022d,Milz2022a} provide asymptotic
consistency results for nonconvex infinite-dimensional stochastic
optimization.
The problem structure given by the compact linear operator in \eqref{eq:intro:nonconvexPDE}
is common among PDE-constrained optimization problems and
has been used for different purposes in the literature.
For example, the authors of \cite{Milz2022a} have employed it to demonstrate
the consistency of empirical estimators for risk-averse stochastic optimization.
Finally, \cite{Milz2022b} establishes nonasymptotic sample size
estimates for risk-neutral semilinear PDE-constrained optimization. In the field of PDE-constrained
optimization under uncertainty, current SAA analyses require certain
control regularizers, such as standard Tikhonov regularizers
\cite{Hoffhues2020,Milz2021}, R-functions \cite{Milz2022a}, and Kadec functions \cite{Milz2022d}. These requirements exclude
the \(L^1\)-norm as a regularizer, for example.
Using compactness of the linear operator  in \eqref{eq:intro:nonconvexPDE},
we are able to extend previous results to general convex control regularizers.

A variety of alternative discretization strategies have been proposed and analyzed for stochastic optimization problems; these include
(randomized) quasi-Monte Carlo methods \cite{Guth2019,Guth2024,Heitsch2016,Koivu2005,Melnikov2025,Pennanen2005}, low-rank tensor approximations 
\cite{Antil2023,Benner2016,Garreis2017},
and stochastic collocation approaches 
\cite{Kouri2013,Tiesler2012,Zahr2019}.
Among these, Monte Carlo-based sampling approaches are the most broadly applicable, though each method offers distinct advantages and limitations.

As discussed in the previous section,  our sample size estimates for convex problems are independent of the compactness of either the feasible set or the operator
$B$. Sample size estimates for stochastic convex optimization without  compactness conditions are also provided in \cite{Liu2024,Milz2021,Shalev-Shwartz2010}.

\subsection{Outline}

\Cref{sec:notation} introduces notation and terminology.
\Cref{sect:risk-neutral-optimal-control} introduces assumptions
on the risk-neutral composite control problems, studies the existence
of solutions, states  optimality conditions, and establishes consistency of SAA optimal values,
solutions, and critical points. \Cref{sect:risk-neutral-optimal-control}
also derives sample size estimates for nonconvex and convex
problems. Our framework is applied to linear and bilinear PDE-constrained
optimization in \Cref{sec:bangbangcontrol}. \Cref{sec:numerics}
presents numerical illustrations. The appendix derives uniform
expectation bounds.

\section{Notation and terminology}
\label{sec:notation}

Throughout the text, normed vector spaces $X$ are defined over the reals. Metric spaces $X$
are equipped with their Borel sigma-field $\mathcal{B}(X)$.
Moreover, we identify the control space $U$, a separable
Hilbert space, with $U^*$, and write $U^* = U$.
For a Banach
space $(X,\norm[X]{\cdot})$ with norm $\norm[X]{\cdot}$, we denote by $X^*$
its topological dual space and by $\dualp[X]{\cdot}{\cdot}$ its
duality pairing. If the norm is clear from the context, we write $X$ instead
of $(X,\norm[X]{\cdot})$. The inner product on a Hilbert space $H$ is denoted
by $\inner[H]{\cdot}{\cdot}$.
A sequence $(x_k) \subset X$ is said to converge 
\emph{weakly} to $x \in X$ if for each  $\ell \in X^*$, 
$\dualp[X]{\ell}{x_k} \to 
\dualp[X]{\ell}{x}
$
as
$k \to \infty$.
If $X$ and $Y$ are Banach spaces with
$X \subset Y$, and $\imath \colon X \to Y$
defined by $\imath x \coloneqq x$ is linear and bounded, then we say that $X$
is \emph{(continuously) embedded} into $Y$. We abbreviate such embeddings by
$X \embedding Y$. A linear operator between two Banach spaces is called
\emph{compact} if its image of the domain space's unit ball is precompact.
For a linear bounded operator $\Upsilon$
between Banach spaces, $\Upsilon^*$
denotes its adjoint operator.
For a measurable space $(\Theta, \mathcal{A})$ and
metric spaces $X$ and $Y$, $h \colon X \times \Theta \to Y$ is
called \emph{\Caratheodory\ mapping} if $h(\cdot, \theta)$ is continuous for all
$\theta \in \Theta$ and $h(x,\cdot)$ is $\mathcal{A}$-$\mathcal{B}(Y)$
measurable for all $x \in X$.
For a bounded domain $\domain \subset \mathbb{R}^d$, we denote by
$L^p(\domain)$ $(1 \leq p \leq \infty)$ the standard
Lebesgue spaces,
and by $H_0^1(\domain)$
and $H^s(\domain)$ ($s \in \mathbb{N}$)
the standard Sobolev spaces.
We set
$H^{-1}(\domain) \coloneqq H_0^1(\domain)^*$.
Let \((Y, d_Y)\) be a metric space, and let \(\nu > 0\).  
A finite set \(\{y_1, y_2, \dots, y_K\} \subseteq Y\) is called a \emph{\(\nu\)-net} of \(Y\) if for every \(y \in Y\), there exists \(k \in \{1, \dots, K\}\) such that
$
d_Y(y, y_k) \le \nu
$.
For a nonempty, compact metric space $Y$, the $\nu$-covering number
$\mathcal{N}(\nu; Y)$, $\nu>0$,
denotes the minimal number of points
in a $\nu$-net of $Y$.

\section{Risk-neutral  composite  optimal control}
\label{sect:risk-neutral-optimal-control}

We start by stating the main assumptions on Problem \eqref{eq:nonconvexriskneutral}.
We recall that
the domain of $\psi$ is denoted by
\begin{align*}
	\adcsp \coloneqq \{\, u \in \csp  \;|\;\psi(u) < \infty \,\}.
\end{align*}

The following assumption imposes the properties of~$\psi$ and  the smoothness and integrability requirements on the integrand.
In the following two assumptions, all integrability conditions are understood with respect to the distribution of $\boldsymbol{\xi}$.

\begin{assumption}[{Integrand and feasible set}]
    \label{assumption:objective_consistency}
    ~
	\begin{enumthm}[wide,nosep,leftmargin=*]
		\item \label{itm:assumption:contspace}
  The control space $\csp$ is a separable Hilbert space, and
        $\adsp$ is a separable Banach space.

		\item\label{itm:assumption:B} The operator
		$B : \csp \to \adsp$ is linear and compact.

		\item\label{itm:assumption:nonsmooth} The function $\psi : \csp \to [0,\infty]$
		  is proper, closed, and convex, and $\adcsp$ is bounded.
        \item The set $\csp_0 \subset \csp$
        is open, convex, and bounded
        with $\adcsp \subset \csp_0$.
        Moreover,
        $f \colon B(\csp_0) \times \Xi \to \mathbb{R}$ is continuous
        in its first argument
        on $B(\adcsp)$
        for each $\xi \in \Xi$
        and measurable in its second
        argument on $\Xi$ for each $w \in B(\csp_0)$.
		\item
        For an integrable random variable
		$\zeta_{\rpobj} : \Xi \to [0,\infty)$, it holds that
  \begin{align*}
      |\rpobj(Bu,\xi)| \leq \zeta_{\rpobj}(\xi) \quad \text{for all} \quad (u,\xi) \in  \csp_{0} \times \Xi.
  \end{align*}
        \end{enumthm}

\end{assumption}
% \change{%
% \Cref{itm:assumption:nonsmooth} differs from the  Carath\'eodory  setting.
% Classically, Carath\'eodory  mappings are defined on the direct product of a complete
% separable metric space and a sample space.
% If \(B(\csp_0)\) were closed and
% continuity property held on all of \(B(\csp_0)\), \Cref{itm:assumption:nonsmooth} 
% would reduce to the
% standard Carath\'eodory condition on \(B(\csp_0)\times\Xi\).
% }

The following assumption formulates
differentiability and integrability statements.

\begin{assumption}[{Integrand and its gradient}]
\label{assumption:gradient_consistency}
~
	\begin{enumthm}[wide,nosep,leftmargin=*]
		\item For each $\xi \in \Xi$, the mapping
    \begin{align*}
        g_{\xi} \colon U_0 \to \mathbb{R}, \quad \text{where} \quad g_{\xi}(u) \coloneqq f(Bu,\xi),
    \end{align*}
    is continuously differentiable.

    \item There exists a Carath\'eodory function~$\Du_w f \colon B(\adcsp) \times \Xi \to W^*$ such that
    \begin{align*}
       \nabla g_\xi(u) = B^* \Du_w f(Bu,\xi) \quad \text{for all} \quad (u,\xi) \in \adcsp \times \Xi.
    \end{align*}
		\item  There exists an integrable random variable
		$\zeta_{\Du \rpobj} : \Xi \to [0,\infty)$ such that
\begin{align*}\norm[U]{\nabla g_\xi(u)}
		&\leq \zeta_{\Du \rpobj}(\xi) \quad \text{for all} \quad  (u,\xi) \in \csp_0 \times \Xi, \quad \quad \text{and} \\
  \norm[\adsp^*]{\Du_w f(Bu,\xi)}
		&\leq \zeta_{\Du \rpobj}(\xi) \quad \text{for all} \quad  (u,\xi) \in \adcsp \times \Xi.
  \end{align*}
	\end{enumthm}
\end{assumption}

If~$f(\cdot,\xi)$ is continuously differentiable on a~$W$-neighborhood of~$B(\adcsp)$, the differentiability of~$ g_\xi(\cdot) = f(B \cdot,\xi)$ and the gradient representation are implied by the chain rule. From this perspective,
\Cref{assumption:gradient_consistency} ensures that the gradients of~$F \circ B $ and~$\hat{F}_N \circ B$  have  composite structure, even when the chain rule does not apply. While this might seem technical at first glance, it allows us to fit challenging settings into the outlined abstract framework. For a particular example, we refer  to the bilinear PDE-constrained problem in  \Cref{subsect:bilinearproblem}.

\subsection{Existence of solutions and optimality conditions}
In this section, we show that both the risk-neutral problem
\eqref{eq:nonconvexriskneutral} and
the associated SAA problems
\eqref{eq:nonconvexsaa} admit solutions.
We also show
the measurability of the SAA optimal value and the existence of measurable SAA
solutions.
Moreover, we introduce the particular form of first-order necessary optimality conditions used in the following sections.

For $\omega \in \Omega$
and $N \in \mathbb{N}$, we denote by
$\hat{\vartheta}^*_N(\omega)$ the optimal value of the SAA problem
\eqref{eq:nonconvexsaa}.

\begin{proposition} \label{lem:objectivesmooth}
    Let \Cref{assumption:objective_consistency} hold, and let $N \in \mathbb{N}$. Then,
    \textnormal{(i)} the risk-neutral problem \eqref{eq:nonconvexriskneutral} and
    for each $\omega \in \Omega$,
    the SAA problem
    \eqref{eq:nonconvexsaa} admit solutions,
    \textnormal{(ii)}
    the function
    $\hat{\vartheta}^*_N
    \colon \Omega \to \mathbb{R}$
    is measurable, and
    \textnormal{(iii)}
    there exists at least one measurable map
    $u^*_N \colon \Omega \to U$
    such that for each $\omega \in \Omega$,
    $u^*_N(\omega)$ solves
    the SAA problem \eqref{eq:nonconvexsaa}.
    \textnormal{(iv)}
    If, moreover,
    \Cref{assumption:gradient_consistency}
	holds, then the composite functions
	$u \mapsto F(Bu)$ and $u \mapsto \hat F_N(Bu)$ are continuously differentiable
    on $U_0$ with gradients
	$u \mapsto  \E{\nabla g_{\boldsymbol{\xi}}(u)}$ and
	$u \mapsto (1/N) \sum_{i=1}^N \nabla g_{\boldsymbol{\xi}^i}(u)$,
    respectively.
\end{proposition}
\begin{proof}
 \textnormal{(i)}
 Since~$\adcsp$ is weakly compact, the proof follows by standard arguments as well as Fatou's lemma.

  \textnormal{(ii)--(iii)}
  Since
    $
        (u,t,\omega) \mapsto \hat{\erpobj}_N(Bu, \omega) + t
    $
    is Carath\'eodory
    on  $\{(u,t) \in
    \adcsp \times \mathbb{R}
    \colon \psi(u) \leq t
    \} \times \Omega$,
    measurability theorems on
    marginal maps and inverse images \cite[Thms.\ 8.2.9 and 8.2.11]{Aubin2009}
    imply the assertions.

    \textnormal{(iv)}
    The assertions can be established
    using standard arguments.
In particular, we can apply Lemma~C.3 from \cite{Geiersbach2020} to deduce 
Fr\'echet differentiability. Using the dominated convergence theorem, we can 
establish the continuity of the Fr\'echet derivatives.
\end{proof}

\Cref{assumption:gradient_consistency,lem:objectivesmooth}
motivate the definitions
\begin{align*}
    \Du F(w) \coloneqq
    \E{\Du_w f(w,\boldsymbol{\xi})}
    \quad \text{and} \quad
    \Du \hat{F}_N(w, \omega)
    \coloneqq \frac{1}{N}
    \sum_{i=1}^N \Du_w f(w,\boldsymbol{\xi}^i(\omega))
\end{align*}
as mappings on $B(\adcsp)$ and $B(\adcsp) \times \Omega$, respectively.
We note that these definitions are formal in the sense that
while $\E{\Du_w f(w,\boldsymbol{\xi})}$ is well-defined, $F$ may not be continuously differentiable, see also the discussion following \Cref{assumption:gradient_consistency}.
As with $\hat{F}_N$, we often omit
the second argument of $\Du \hat{F}_N$.

Furthermore, we define the \emph{gap functional}
$\Psi : \adcsp \to [0,\infty]$ 
\begin{align}
\label{def:gapfunctional}
\begin{aligned}
\Psi(u) \coloneqq \sup_{v \in \adcsp} \, \lbrack \, \inner[\csp]{ B^* \Du F(Bu)}{u-v} +\psi(u)-\psi(v) \, \rbrack,
\end{aligned}
\end{align}
and recall that its SAA counterpart, $\hat{\Psi}_N$, is defined in \eqref{eq:SAAgapfunctional}.
A point $\bar u \in \adcsp$ is referred to as \emph{critical point}
of \eqref{eq:nonconvexriskneutral} if
$- B^* \Du F(B\bar u) \in \partial \psi(\bar u)$.
Similarly, a point $\bar u_N \in \adcsp$ is called a \emph{critical point}
of the SAA problem \eqref{eq:nonconvexsaa} if
$- B^*\Du \hat{F}_N(\bar u_N) \in \partial \psi(\bar u_N)$. Critical points of \eqref{eq:nonconvexriskneutral} are the zeros
of the gap function $\Psi$ as
summarized in the following proposition.

\begin{proposition} \label{prop:propertiesofgap}
	If \Cref{assumption:objective_consistency,assumption:gradient_consistency}
	hold, then
	\emph{(i)}
	an element~$\bar u \in \adcsp$ is a critical point
	of \eqref{eq:nonconvexriskneutral} if and only if $\Psi(\bar u) = 0$,
     \emph{(ii)}
    a point $\bar u_N \in \adcsp$ is a critical point
    of \eqref{eq:nonconvexsaa} if and only if $\hat{\Psi}_N(\bar u_N) = 0$, and
	\emph{(iii)}
    $\Psi$ is weakly lower semicontinuous.
\end{proposition}
\begin{proof}
	Parts (i) and (ii) can be established using arguments similar to those used
    to prove Theorem~2.4 in \cite{Kunisch2021}.
    Part (iii) follows from Theorem~2.5~(c) in \cite{Herzog2012}
    because  $B$ is completely continuous.
\end{proof}

\subsection{Consistency of SAA optimal values and solutions}

We establish the asymptotic consistency
of SAA optimal values and
the weak consistency of SAA solutions.
We denote by $\vartheta^*$ the optimal value
of the true problem \eqref{eq:nonconvexriskneutral}.

\begin{theorem}
\label{thm:nonconvex_consistency_solutions}
	If \Cref{assumption:objective_consistency} holds, then
    \textnormal{(i)}
	$\hat{\vartheta}_N^* \to \vartheta^*$ as $N \to \infty$
	w.p.~$1$, and
    \textnormal{(ii)}
    for almost every $\omega \in \Omega $, $(u^*_N(\omega))$ has
    at least one weak accumulation point and each
    such point solves
	\eqref{eq:nonconvexriskneutral}.
\end{theorem}

\begin{proof}
    Before we establish parts (i) and (ii), let us note that the uniform
    law of large numbers \cite[Cor.\ 4:1]{LeCam1953}
    ensures the
    existence of a null set
	$\Omega_0 \subset \Omega$ with $\Omega_0 \in \mathcal{F}$
	such that for all $\omega \in \Omega \setminus \Omega_0$, we have
    \begin{align}
        \label{eq:ULLN:objective}
        \sup_{w \in B(\adcsp)} |\erpobj(w) - \hat{\erpobj}_N(w,\omega)| \to 0
        \quad \text{as} \quad N \to \infty.
    \end{align}
    Moreover,
	$\erpobj$ is continuous on $B(\adcsp)$ owing
    to the dominated convergence theorem.

    \textnormal{(i)}
    We have for all
    $\omega \in \Omega$,
    \begin{align*}
        |\hat{\vartheta}_N^*(\omega) - \vartheta^*|
        \leq \sup_{u \in \adcsp}
	|\erpobj(Bu) + \psi(u) - \big(\hat{\erpobj}_N(Bu,\omega) + \psi(u)\big)|
	\leq
	\sup_{w \in B(\adcsp)} |\erpobj(w) - \hat{\erpobj}_N(w,\omega)|.
    \end{align*}
	Combined with \eqref{eq:ULLN:objective}, we find that
	$\hat{\vartheta}_N^*(\omega) \to \vartheta^*$ as $N \to \infty$
	for all $\omega \in \Omega \setminus \Omega_0$.

    \textnormal{(ii)}
    Fix $\omega \in \Omega \setminus \Omega_0 $.
    Since $(u_N^*(\omega))$
    is bounded,
    it has a weak accumulation point.
    Let $\bar u(\omega) \in \csp$ be such a point.
	Hence, there exists a subsequence
    $(u_N^*(\omega))_{\mathcal{M}(\omega)}$ of
    $(u_N^*(\omega))$ such that
	$u_N^*(\omega) \wto
	\bar u(\omega) $ as $\mathcal{M}(\omega) \ni N \to \infty$.
    Here, \(\mathcal{M}(\omega) \subseteq \mathbb{N}\) is an infinite subset of natural numbers, potentially depending on \(\omega \in \Omega \setminus \Omega_0\).
  	Fix $u \in \adcsp$. We have
	\begin{align*}
	\hat{\vartheta}_N^*(\omega)
	= \hat{\erpobj}_N(Bu_N^*(\omega),\omega) + \psi(u_N^*
    (\omega))
	\leq  \hat{\erpobj}_N(Bu,\omega) + \psi(u).
	\end{align*}
	Using
    the uniform convergence statement in
    \eqref{eq:ULLN:objective},
    the continuity of $F$,
    and
	$Bu_N^*(\omega) \to B\bar u(\omega)$
	as $\mathcal{M}(\omega) \ni N \to \infty$, we obtain
	$$\hat{\erpobj}_N(Bu_N^*(\omega)) = \big[\hat{\erpobj}_N(Bu_N^*(\omega)) -\erpobj(Bu_N^*(\omega))\big] + \erpobj(Bu_N^*(\omega)) \to \erpobj(B\bar u(\omega)) $$
	and $\hat{\erpobj}_N(Bu) \to \erpobj(Bu)$
	as $\mathcal{M}(\omega) \ni N \to \infty$.
	Combined with the fact that
    $\psi$ is weakly lower semicontinuous,
	\begin{align*}
	\erpobj(B \bar u(\omega)) + \psi(\bar u(\omega))
	& \leq  \liminf_{\mathcal{M}(\omega) \ni N \to \infty}
	\{\, \hat{\erpobj}_N(Bu_N^*(\omega)) +
	 \psi(u_N^*(\omega)) \, \}
     \\
	& \leq \lim_{\mathcal{M}(\omega) \ni N \to \infty}\,
	\hat{\erpobj}_N(Bu) + \psi(u)
	= \erpobj(Bu) + \psi(u).
	\end{align*}
    Since $\Omega \setminus \Omega_0  \in \mathcal{F}$ occurs
    w.p.~$1$, we obtain the assertion.
\end{proof}

\subsection{Consistency of SAA critical points}

We establish the asymptotic consistency of the SAA gap functional
evaluated at SAA critical
points and that of
SAA critical points.

\begin{theorem}
	\label{thm:consistency_critical_points}
	Let \Cref{assumption:objective_consistency,assumption:gradient_consistency}
	hold,  and let $(\varepsilon_N) \subset [0,\infty)$ be a
    sequence with $\varepsilon_N \to 0$ as $N \to \infty$.
    Let $(\bar u_N) \subset\adcsp$ be a sequence of
	random vectors with $\hat{\Psi}_N(\bar u_N)  \leq \varepsilon_N$ w.p.~$1$.
    Then
	\textnormal{(i)}  $\Psi(\bar u_N) \to 0$ as $N \to \infty$ w.p.~$1$.
	\textnormal{(ii)}
    For almost every $\omega \in \Omega $, $(\bar u_N(\omega))$ has at least one weak accumulation point
    and every such point is a zero of $\Psi$.
\end{theorem}

\begin{proof}
    Before we establish parts (i) and (ii),
    let us note that for all $u \in \adcsp$,
     \begin{align}
    \label{eq:consthelpforgap'}
        \Psi(u)
        & \leq \hat{\Psi}_N(u) + \sup_{v \in \adcsp}\,
         \inner[\csp]{B^* \Du \erpobj (Bu)-B^* \Du \hat{\erpobj}_N (Bu)}{u-v}.
    \end{align}
    Hence, we have
    for all $u \in \adcsp$,
	\begin{align}
    \label{eq:consthelpforgap}
	\begin{aligned}
	    \Psi(u) & \leq  \hat{\Psi}_N(u) + \sup_{(v,w) \in B(\adcsp) \times B(\adcsp)}\,
         \dualp[W]{\Du \erpobj (w)-\Du \hat{\erpobj}_N (w)}{w-v}.
	\end{aligned}
	\end{align}
    Therefore,
    the uniform
    law of large numbers \cite[Cor.\ 4:1]{LeCam1953}
    ensures the
    existence of a null set
	$\Omega_0 \subset \Omega$ with $\Omega_0 \in \mathcal{F}$
	such that for all $\omega \in \Omega \setminus \Omega_0$, we have
    \begin{align*}
        \sup_{(v,w) \in B(\adcsp)
        \times B(\adcsp)}\,
         \dualp[W]{\Du \erpobj (w)-\Du \hat{\erpobj}_N (w,\omega)}{w-v}
         \to 0 \quad \text{as} \quad N \to \infty.
    \end{align*}

    \textnormal{(i)}
    Using $\hat{\Psi}_N(\bar u_N) \leq \varepsilon_N$,
	   the error bound \eqref{eq:consthelpforgap},
    and  the above convergence statement, we obtain
	$\Psi(\bar u_N) \to 0$ as $N \to \infty$ w.p.~$1$.

    \textnormal{(ii)}
    Fix $\omega \in \Omega \setminus
    \Omega_0$ such that
    $\Psi(\bar u_N(\omega)) \to 0$ as $N \to \infty$.
    Since $(\bar u_N(\omega))$
    is bounded, it has a weak accumulation point.
	If $\bar u_N(\omega) \wto \bar u(\omega)$ as $\mathcal{M}(\omega) \ni N \to \infty$, then the weak lower semicontinuity of $\Psi$
	(see \Cref{prop:propertiesofgap})
    and $\Psi(\bar u_N(\omega)) \to 0$ as $N \to \infty$ imply the assertion.
    Here, \(\mathcal{M}(\omega) \subseteq \mathbb{N}\) is again an infinite subset of natural numbers, potentially depending on \(\omega \in \Omega \setminus \Omega_0\).
\end{proof}

\subsection{Sample size estimates for SAA critical points}
\label{subsect:sample_size_estimates}

We establish expectation bounds and sample size estimates. Our analysis is inspired by that of \cite{Shapiro2021}.

For this purpose, we formulate a Lipschitz continuity assumption and a light-tailed condition on the integrand's gradient,
which are 
typical assumptions in related contexts.
Specifically, for
Lipschitz-type assumptions in the context of sample complexity,  
we refer the reader to
section~9.2.11 of \cite{Shapiro2021}.
Our light-tailed condition is inspired by that used in  
eq.~(2.50) in \cite{Nemirovski2009}; see also eq.~(4.1.15) in \cite{Lan2020}. 

\begin{assumption}[{Lipschitz condition
on integrand's derivative, sub-Gaussian derivatives}]
    \label{assumption:lipschitz-subgaussian}

    \begin{enumthm}[wide,nosep,leftmargin=*]
        \item
        \label{assumption:Lipschitz_continuity}
    For an integrable  random variable
	$L_{\Du f} : \Xi \to [0,\infty)$,
	\begin{align*}
	    \norm[\csp]{\nabla g_\xi(u_2)- \nabla g_\xi(u_1)}
	    \leq L_{\Du f}(\xi)\norm[\adsp]{Bu_2-Bu_1}
	    \quad \text{for all} \quad u_1, u_2 \in \adcsp, \quad \xi \in \Xi.
	\end{align*}
    \item \label{assumption:subgaussian_gradients}
    For some constant $\tau_{\Du f} > 0$,
        \begin{align*}
            \E{\exp(\tau_{\Du f}^{-2}\norm[\csp]{\nabla g_{\boldsymbol{\xi}}(u) - \E{\nabla g_{\boldsymbol{\xi}}(u) }}^2)} \leq \mathrm{e}
            \quad \text{for all} \quad
            u \in \adcsp.
        \end{align*}
    \end{enumthm}
\end{assumption}

\Cref{assumption:subgaussian_gradients} is, for example, satisfied
if there exists a constant $C > 0$ such that
$\|\nabla g_\xi(u)\|_\csp \leq C$ for all $(u,\xi) \in \adcsp \times \Xi$.
We denote by $r_{\mathrm{ad}}$ the diameter
of $\adcsp$. We define 
$\bar{L}_{\Du f} \coloneqq \mathbb{E}[L_{\Du f}(\boldsymbol{\xi})]$.

\begin{theorem}
    \label{thm:sample_size_estimates_gap_function}
    Let \Cref{assumption:objective_consistency,assumption:gradient_consistency,assumption:lipschitz-subgaussian} hold.
    For each $N \in \mathbb{N}$, let $u_N \colon \Omega \to \adcsp$ be measurable.
    Then, for all $N \in \mathbb{N}$ and $\nu > 0$,
    \begin{align}
        \label{eq:sample_size_estimates_gap_function}
        \E{\Psi(u_N)} \leq \E{\hat{\Psi}_N(u_N)} +
        2\bar{L}_{\Du f} r_{\mathrm{ad}} \nu + \tfrac{\sqrt{3}\tau_{\Du f} r_{\mathrm{ad}} }{\sqrt{N}} \big(\ln(2\mathcal{N}(\nu; B(\adcsp)))\big)^{1/2}.
    \end{align}
\end{theorem}

\begin{proof}
    Using \eqref{eq:consthelpforgap'}, we have
    \begin{align*}
	\E{\Psi(u_N)} \leq   \E{\hat{\Psi}_N(u_N)}
	+ r_{\mathrm{ad}}  \E{\sup_{v \in B(\adcsp)} \, \norm[\csp]{B^* \Du F (v)- B^* \Du \hat{F}_N (v)}}.
	\end{align*}
    Applying \Cref{prop:uniformexpectationbounds}
    to $\mathsf{G}(v,\xi) \coloneqq B^* \Du_w f(v,\xi)$, we obtain
    \begin{align*}
        \E{\sup_{v \in B(\adcsp)} \norm[\csp]{B^* \Du F (v)- B^* \Du \hat{F}_N (v)}}
        \leq 2\bar{L}_{\Du f}\nu + \tfrac{\sqrt{3}\tau_{\Du f}}{\sqrt{N}} \big(\ln(2\mathcal{N}(\nu; B(\adcsp)))\big)^{1/2}.
    \end{align*}
\end{proof}

Using \Cref{thm:sample_size_estimates_gap_function}, we establish
sample size estimates.

\begin{corollary}
    \label{cor:sample_size_estimates_gap_function}
    Let the hypotheses of \Cref{thm:sample_size_estimates_gap_function}
    hold.
    If $\varepsilon > 0$ and
    \begin{align*}
        N \geq \frac{12 r_{\mathrm{ad}}^2 \tau_{\Du f}^2}{\varepsilon^2} \cdot
        \ln\big(2\mathcal{N}(\varepsilon/(4 r_{\mathrm{ad}}  \bar{L}_{\Du f}); B(\adcsp))\big),
    \end{align*}
    then
    $
        \E{\Psi(u_N)} \leq \E{\hat{\Psi}_N(u_N)}  + \varepsilon
    $.
\end{corollary}
\begin{proof}
    We use \eqref{eq:sample_size_estimates_gap_function}
    with $\nu = \varepsilon/(4 r_{\mathrm{ad}} \bar{L}_{\Du f})$.
    We have $2\bar L_{\Du f}r_{\mathrm{ad}}\nu = \varepsilon/2$. The choice
    of $N$ ensures that the second term in the right-hand side
in \eqref{eq:sample_size_estimates_gap_function} does not exceed
$\varepsilon/2$.
\end{proof}

The sample size estimates in \Cref{cor:sample_size_estimates_gap_function} are based on
the covering numbers of $B(\adcsp)$. A typical example in PDE-constrained optimization
models $B$ as the adjoint operator of $H^1(\domain) \embedding L^2(\domain)$.
Combining the covering numbers for Sobolev function
    classes in \cite[Thm.\ 1.7]{Birman1980} and duality of metric entropy \cite[p.\ 1315]{Artstein2004},
covering numbers of $B(\adcsp)$ can be derived.
Specifically, let $\imath$ be the embedding operator of the embedding 
$H^1(\domain) \embedding L^2(\domain)$. Throughout the paragraph, we set $D = (0,1)^d$. 
Let 
$\bar{\mathbb{B}}_{H^1(\domain)}(0)$ denote the closed $H^1(\domain)$-unit ball. Then, 
Theorem~1.7 in \cite{Birman1980} implies that 
\[
\ln \mathcal{N}(\nu; \imath(\bar{\mathbb{B}}_{H^1(\domain)}(0)))
\]
scales like $(1/\nu)^d$
for all sufficiently small $\nu > 0$. 
Now, let $B$ be the adjoint operator of $\imath$, and let $\psi$ be the 
indicator function of the closed unit ball in $L^2(D)$. Since $H^1(D)$ and $L^2(D)$ are 
Hilbert spaces, 
the above covering numbers and
duality of metric entropy \cite[p.\ 1315]{Artstein2004} 
ensure that 
\[
\ln \mathcal{N}(\nu; B(\adcsp))
\]
grows with $(1/\nu)^d$  for all sufficiently
small $\nu > 0$. These covering numbers inform the sample size estimate in \Cref{cor:sample_size_estimates_gap_function}.

\subsection{Expectation bounds for convex problems}
\label{subsec:expectationboundsconvexproblems}

For convex problems, we demonstrate nonasymptotic expectation bounds
for the distance between SAA solutions and the true solution, the gap functional, and the optimality gap. These statements hold true without the compactness of the linear bounded operator $B$. In the following, let $u^*$ be a solution to \eqref{eq:nonconvexriskneutral}
and for each $N \in \mathbb{N}$, let $u_N^*$ be a measurable solution to \eqref{eq:nonconvexsaa}.

\begin{assumption}[{H\"older-type inequality, gradient regularity, growth, variance}]
    \label{assumption:hoelder-growth-deviation}
    ~
    \begin{enumthm}[wide,nosep,leftmargin=*]
    \item
    \label{assumption:hoelder}
    The space $\mathcal{U}$ is
    a Banach space with $\mathcal{U}^* \embedding \csp \embedding \mathcal{U}$, $\mathcal{H}$ is a
    separable Hilbert space
    with $\mathcal{H} \embedding \mathcal{U}^*$, and
    \begin{align*}
        |\inner[\csp]{v}{u}|
        \leq \norm[\mathcal{U}^*]{v} \norm[\mathcal{U}]{u}
        \quad \text{for all} \quad v \in \mathcal{U}^*,
        \quad u \in U.
    \end{align*}
    \item
    \label{assumption:nonconvex:growth}
    For two constants $\theta \in (0,\infty)$
	and $\varrho \in [2,\infty)$, it holds that
	\begin{align*}
	\inner[\csp]{B^*\Du\erpobj(Bu^*)}{u-u^*}
	+ \psi(u) - \psi(u^*) \geq
	\theta \norm[\mathcal{U}]{u-u^*}^{\varrho}
	\quad \text{for all} \quad u \in \csp.
	\end{align*}
    \item
    The mapping $\Xi \ni \xi \mapsto \nabla g_\xi(u^*) \in \mathcal{H}$
    is measurable, $\nabla g_{\boldsymbol{\xi}}(u^*)$ is integrable,
    and the standard deviation-type constant
    \begin{align*}
    \sigma_{\Du \rpobj} \coloneqq
        (\E{\norm[\mathcal{H}]{\nabla g_{\boldsymbol{\xi}}(u^*)-
        		\mathbb{E}[\nabla g_{\boldsymbol{\xi}}(u^*)]}^2})^{1/2}
    \end{align*}
    is finite.
    \end{enumthm}
\end{assumption}
\Cref{assumption:nonconvex:growth}
and its variants have been utilized, for example,
in \cite{Kunisch2021} and, in particular,
ensures that
$u^*$ is the unique solution to
\eqref{eq:nonconvexriskneutral}.

We state the section's main result.
If \Cref{assumption:hoelder} holds true,
let $C_{\mathcal{H}; \mathcal{U}^*} \in (0,\infty)$ be the embedding
constant of the embedding $\mathcal{H} \embedding \mathcal{U}^*$.
If  $\varrho = 2$ in \Cref{assumption:hoelder-growth-deviation},
the following result ensures the
typical Monte Carlo convergence rate, $1/\sqrt{N}$,
for the expected error $\E{\norm[\mathcal{U}]{u_N^*-u^*}}$.

\begin{theorem}
	\label{thm:poisson:expectationbound}
	If \Cref{assumption:objective_consistency,assumption:gradient_consistency,assumption:hoelder-growth-deviation} hold,
	and $\rpobj(\cdot,\xi)$ is convex for all $\xi \in \Xi$, then
	for all $N \in \mathbb{N}$,
	\begin{align*}
	\theta^2 \E{\norm[\mathcal{U}]{u_N^*-u^*}^{2(\varrho-1)}}
	\leq (1/N)C_{\mathcal{H}; \mathcal{U}^*}^2\sigma_{\Du \rpobj}^2.
	\end{align*}
\end{theorem}

\begin{proof}
    The proof is inspired by those of
	Lemma~6 and Theorem~3 in \cite{Milz2021}.
    Using the optimality condition
    $-B^*\Du\erpobj(Bu^*) \in \partial \psi(u^*)$
	and \Cref{assumption:nonconvex:growth}, we obtain
	\begin{align*}
	\inner[\csp]
	{B^*\Du \erpobj(Bu^*)-B^*\Du \hat{\erpobj}_N(Bu_N^*)}{u_N^*-u^*}
	\geq \theta
	\norm[\mathcal{U}]{u_N^*-u^*}^{\varrho}.
	\end{align*}
	Since $u \mapsto \hat{\erpobj}_N(Bu)$ is continuously differentiable
    on $\csp_0$ (see \Cref{lem:objectivesmooth}),
	and  convex,
	\begin{align*}
	\inner[\csp]{B^*\Du \hat{\erpobj}_N(Bu_N^*)-B^*\Du \hat{\erpobj}_N(Bu^*)}{u_N^*-u^*} \geq 0.
	\end{align*}
	Adding both estimates ensures
    \begin{align*}
	\inner[\csp]{B^*\Du \erpobj(Bu^*)-B^*\Du \hat{\erpobj}_N(Bu^*)}{u_N^*-u^*}
	\geq \theta
	\norm[\mathcal{U}]{u_N^*-u^*}^{\varrho}.
	\end{align*}
    Hence
	\begin{align*}
	\norm[\mathcal{U}^*]
	{B^*\Du \erpobj(Bu^*)-B^*\Du \hat{\erpobj}_N(Bu^*)}
	\geq \theta
	\norm[\mathcal{U}]{u_N^*-u^*}^{\varrho-1}.
	\end{align*}
	Taking squares and
	using the continuity of the embedding
	$\mathcal{H} \embedding \mathcal{U}^*$,
	we obtain
	\begin{align*}
	C_{\mathcal{H}; \mathcal{U}^*}^2 \norm[\mathcal{H}]
	{B^*\Du \erpobj(Bu^*)-B^*\Du \hat{\erpobj}_N(Bu^*)}^2
	\geq \theta^2
	\norm[\mathcal{U}]{u_N^*-u^*}^{2(\varrho-1)}.
	\end{align*}
	Since $\E{B^*\Du_w \rpobj(Bu^*,\boldsymbol{\xi})} = B^*
    \Du \erpobj(Bu^*)$
    (see \Cref{lem:objectivesmooth}),
	$B^*\Du_w \rpobj(Bu^*,\boldsymbol{\xi}^i)$, $i = 1, 2, \ldots$
	are independent and identically distributed,
	and $\mathcal{H}$ is a
	separable Hilbert space, we obtain
	\begin{align}
    \label{eq:expectationboundH}
		\E{\norm[\mathcal{H}]
			{B^*\Du \erpobj(Bu^*)-B^*\Du \hat{\erpobj}_N(Bu^*)}^2}
		= (1/N)\sigma_{\Du \rpobj}^2.
	\end{align}
	This implies the expectation bound.
\end{proof}

Using this result, we further deduce superconvergence properties of the objective function values and the gap functional.
We recall that $\objective$ denotes the objective
function of \eqref{eq:nonconvexriskneutral}.

\begin{corollary} \label{coroll:fastrateresandgap}
    	Let the hypotheses of \Cref{thm:poisson:expectationbound} hold with~$\varrho=2$, and suppose that there exists
	$L_{\Du f, \mathcal{U}} >0$ such that
	\begin{align} \label{ass:smootherlipest}
	    \norm[\mathcal{U}^*]{\nabla g_\xi(u_2)- \nabla g_\xi(u_1)}
	    \leq L_{\Du f, \mathcal{U}}\norm[\mathcal{U}]{u_2-u_1}
	    \quad \text{for all} \quad u_1, u_2 \in \adcsp, \quad \xi \in \Xi.
	\end{align}
 Then, for all~$N \in \mathbb{N}$, we have
     \begin{align*}
       \E{\objective(u_N^*) - \objective(u^*)}
       \leq \E{\Psi(u^*_N)}  \leq
       \bigg(1+\bigg(\frac{1}{\theta}+\frac{L_{\Du f, \mathcal{U}} }{\theta^2}\bigg)^2
       +
       \frac{4L_{\Du f, \mathcal{U}}^2}{\theta^2}
       \bigg)
\frac{C_{\mathcal{H}; \mathcal{U}^*}^2\sigma^2_{\Du \rpobj}}{N}.
     \end{align*}
\end{corollary}
\begin{proof}
    Since $\objective(u) - \objective(u^*)\leq \Psi(u)$ for all~$u\in \adcsp$ (see, for example, \cite[Thm.\ 2.4]{Kunisch2021}), it suffices to show the second estimate. For this purpose, note that there exists $v_N \in \adcsp$ with
    \begin{align*}
       \Psi(u^*_N)&= \inner[\csp]{ B^* \Du F(Bu^*_N)}{u^*_N-v_N} +\psi(u^*_N)-\psi(v_N)
\\
& \leq \inner[\csp]{ B^* \Du F(Bu^*_N)-B^* \Du \hat{F}_N (Bu^*_N)}{u^*_N-v_N}
       \\
       &\leq (1/2)\norm[\mathcal{U}^*]{ B^* \Du F(Bu^*_N)-B^* \Du \hat{F}_N (Bu^*_N)}^2
       +(1/2)\norm[\mathcal{U}]{u^*_N-v_N}^2,
    \end{align*}
    where the first inequality follows from the optimality of $u_N^*$. Using \eqref{ass:smootherlipest}, we have
    \begin{align*}
    \norm[\mathcal{U}^*]{ B^* \Du F(Bu^*_N)-B^* \Du \hat{F}_N (Bu^*_N)} &\leq 2 L_{\Du f, \mathcal{U}} \norm[\mathcal{U}]{u^*_N-u^*}
\\& \quad + C_{\mathcal{H}; \mathcal{U}^*} \norm[\mathcal{H}]{B^* \Du F(Bu^*)-B^*\Du \hat{F}_N (Bu^*)}.
    \end{align*}
    Following the proof of Lemma~4.8 in \cite{Kunisch2021}, we arrive at
    \begin{align*}
\theta
	\norm[\mathcal{U}]{v_N-u^*}^2 &\leq	\inner[\csp]
	{B^*\Du \erpobj(Bu^*)-B^*\Du \erpobj(Bu^*_N)}{v_N-u^*}
\\
&\leq  \norm[\mathcal{U}^*]{B^*\Du \erpobj(Bu^*)-B^*\Du \erpobj(Bu^*_N)} \norm[\mathcal{U}]{v_N-u^*} \\
 & \leq L_{\Du f, \mathcal{U}} \norm[\mathcal{U}]{u^*_N-u^*} \norm[\mathcal{U}]{v_N-u^*},
	\end{align*}
 where the final inequality is due to~\eqref{ass:smootherlipest}.
 Putting together the pieces, we have
 \begin{align*}
     \Psi(u^*_N)
     &\leq
     \frac{1}{2}\bigg(1+\frac{L_{\Du f, \mathcal{U}} }{\theta}\bigg)^2
     \norm[\mathcal{U}]{u^*_N-u^*}^2
     +
     4 L_{\Du f, \mathcal{U}}^2 \norm[\mathcal{U}]{u^*_N-u^*}^2
\\
& \quad + C_{\mathcal{H}; \mathcal{U}^*}^2 \norm[\mathcal{H}]{B^* \Du F(Bu^*)-B^*\Du \hat{F}_N (Bu^*)}^2.
 \end{align*}
Taking expectations, and using \eqref{eq:expectationboundH} and \Cref{thm:poisson:expectationbound},
we obtain the expectation bound.
\end{proof}

\section{Application to  bang-bang optimal control under uncertainty}
\label{sec:bangbangcontrol}
We apply our results to risk-neutral, nonsmooth PDE-constrained minimization problems of the form
\begin{align}	\label{eq:riskneutralpdeproblem}
	\min_{u\in L^2(\domain)}\,
	\E{J( \imath y(\boldsymbol{\xi}))} + \beta \norm[L^1(\domain)]{u} \,
    \quad \text{s.t.} \quad u \in [\lb, \ub],
\end{align}
where $\domain \subset \mathbb{R}^d$,~$d\in\{1,2,3\}$, is a bounded, polyhedral domain,~$\beta \geq 0$, $J$ is a sufficiently smooth fidelity term defined on~$L^2(\domain)$ and $\imath \colon H^1_0 (D) \to L^2(D)$ denotes the compact embedding. Moreover,
\begin{align*}
    [\lb, \ub] \coloneqq
\{ \, v \in L^2(\domain)  \;|\; \lb(x) \leq v(x) \leq \ub(x) \text{ for a.e. } x \in \domain \, \},
\end{align*}
where $\lb, \ub \in L^\infty(\domain)$ with  $\lb(x) \leq 0 \leq \ub(x) $ for a.e.\ $x \in \domain$.
The deterministic control
$u \in [\lb, \ub]$
is coupled with the random state variable~$y(\xi)\in H^1_0(\domain)$ via a parametrized elliptic PDE
\begin{align}
\label{eq:bilinear}
\inner[L^2(\domain)^d]{\kappa(\xi) \nabla y}{\nabla v}
+ \inner[L^2(\domain)]{E(u,y)}{v}
= \inner[L^2(\domain)]{b(\xi)}{v}
\quad \text{for all} \quad v \in H_0^1(\domain),
\end{align}
which the state satisfies for every~$\xi\in\Xi$.
Throughout this section, we consider
\begin{align*}
  U=L^2(\domain), \quad W=H^{-1}(\domain), \quad \mathcal{H}=H^2(\domain), \quad \mathcal{U}= L^1(\domain), \quad \text{and} \quad B=\imath^*.
\end{align*}

In the following, we demonstrate that both linear control problems, where \( E(u,y) = -u \), and bilinear control problems, where \( E(u,y) = uy \), can be incorporated into the theoretical framework defined by~\eqref{eq:nonconvexriskneutral}. This is achieved by defining the penalty term as
\[
    \psi(u) = \beta \norm[L^1(\domain)]{u} + I_{[\lb, \ub]}(u),
\]
and by introducing the integrand
 \begin{align} \label{eq:integrandpdeex}
     f(w,\xi)\coloneqq J(B^* S(w,\xi)), \quad \text{where} \quad S \colon B(U_0) \times \Xi \to H^1_0(\domain).
 \end{align}
Here, \( S \colon B(U_0) \times \Xi \to H^1_0(\domain) \) is a suitable parametrized control-to-state operator, defined on the image of an \( L^2(\domain) \)-neighborhood \( U_0 \) of the admissible control set \( \adcsp = [\lb, \ub] \). For both problem classes, the following assumptions are made.
\begin{assumption}[{Fidelity term, PDE right-hand side, random diffusion coefficient}]
\label{assumption:poisson:objective}
 ~
	\normalfont
	\begin{enumthm}[nosep,leftmargin=*]
		\item
        \label{itm:application_objective_lipschitz_gradient}
        The function $\pobj : L^2(\domain)  \to [0,\infty)$
		is convex and continuously differentiable. Its gradient  is Lipschitz continuous with Lipschitz
		constant $\ell \in (0,\infty)$.
        Moreover
        $
            \pobj(v) \leq \varrho(\norm[L^2(\domain)]{v})
        $
        for all $v \in L^2(\domain)$,
        where
		$\varrho : [0,\infty) \to [0,\infty)$ is a polynomial and nondecreasing.
         \item
        The map $b : \Xi \to L^2(\domain)$ is measurable
        and there exists $b_{\sup}^* > 0$
        such that $\|b(\xi)\|_{L^2(D)}
        \leq b_{\sup}^*$ for all $\xi \in \Xi$.
        \item
    	\label{assumption:poisson}
    	The mapping $\kappa : \Xi \to C^1(\bar{\domain})$
    	is  measurable
        and $\E{\|\kappa(\boldsymbol{\xi})\|_{C^1(\bar \domain)}^p} < \infty$
        for all $p \in [1,\infty)$.
        Moreover, there exists $\kappa_{\inf}^* > 0$ such that~$\kappa(\xi)(x) \geq \kappa_{\inf}^*$
        for all~$x \in \domain$ and~$\xi \in \Xi$.
	\end{enumthm}
\end{assumption}

We point out that \Cref{itm:application_objective_lipschitz_gradient} implies
	\begin{align}
    \label{eq:stability_estimate_gradient}
	\norm[L^2(\domain)]{\nabla \pobj(v)}^2
	\leq 2 \ell \pobj(v) \leq 2 \ell \varrho(\norm[L^2(\domain)]{v}) \
	\quad \text{for all} \quad v \in L^2(\domain).
	\end{align}
Moreover,~\Cref{itm:assumption:contspace,itm:assumption:B,itm:assumption:nonsmooth} are clearly satisfied. For abbreviation, let~$C_\domain$ be the Friedrichs constant of the domain $\domain$, which equals the operator norm of $\imath$.
\Cref{assumption:poisson} guarantees $H^2(D)$-regularity of the objective function's gradient. This justifies our choice 
$\mathcal{H} = H^2(D)$ in \Cref{assumption:hoelder}, where $\mathcal{H}$ is 
required to be a separable Hilbert space.
We recall that this Hilbert space structure yields an identity for the variance 
of the difference between the empirical and true gradients, as shown in 
\eqref{eq:expectationboundH}.

\subsection{Risk-neutral affine-linear PDE-constrained optimization}
\label{subsect:bangbanglinear}

We first consider \eqref{eq:riskneutralpdeproblem} governed by a class of affine-linear, parameterized PDEs, that is, $E(u,y)=-u$. For simplicity, set~$b=0$. We consider the operator
\begin{align*}
    A \colon C^1(\bar{\domain}) \times H^{-1}(\domain) \times H^1_0(\domain) \to H^{-1}(\domain),
\end{align*}
where
\begin{align*}
    \dualpHzeroone[\domain]{A(\kappa,w ,y)}{v}\coloneqq \inner[L^2(\domain)^d]{\kappa \nabla y}{\nabla v}
- \dualpHzeroone[\domain]{w}{v}
\quad \text{for all} \quad v \in H_0^1(\domain),
\end{align*}
and the associated parametrized equation
\begin{align} \label{eq:linearstrong}
\text{find}~y \in H^1_0(\domain) \colon \quad A(\kappa, w,y)=0, \quad \text{where} \quad (\kappa, w) \in \mathfrak{P}_{\text{ad}},
\end{align}
and
\begin{align*}
\mathfrak{P}_{\text{ad}} \coloneqq  \big\{\,\kappa  \in C^1(\bar{\domain})\;|\;\kappa > \kappa^*_{\inf}/2\,\big\} \times H^{-1}(\domain).
\end{align*}
With slight abuse of notation, we denote by $\kappa$ in \eqref{eq:linearstrong}
a generic element in 
$\{\,\kappa  \in C^1(\bar{\domain})\;|\;\kappa > \kappa^*_{\inf}/2\,\}$, 
rather than the random diffusion coefficient  in \Cref{assumption:poisson}.
The corresponding solution is denoted by 
\(y(\kappa, w)\). In particular, \eqref{eq:linearstrong} is a deterministic, 
parameterized operator equation with  parameters \((\kappa, w)\). 
When the random diffusion coefficient from \Cref{assumption:poisson}, evaluated 
at \(\boldsymbol{\xi}\), is inserted into \(y(\cdot, w)\), the map 
\(y(\kappa(\boldsymbol{\xi}, w))\) becomes  random.

The next lemma follows from the Lax--Milgram lemma, the implicit function theorem, and standard regularity results for elliptic PDEs. Its proof is omitted for the sake of brevity.
\begin{lemma}  \label{lem:extlinpde}
    Let \Cref{assumption:poisson:objective}
    hold with $b = 0$.
    For every~$(\kappa,w) \in \mathfrak{P}_{\text{ad}}$, there exists a unique~$y=y(\kappa,w) \in H^1_0 (\domain)$ satisfying~\eqref{eq:linearstrong}. The  mapping~$\mathfrak{P}_{\text{ad}} \ni (\kappa,w) \mapsto y(\kappa,w)$
is infinitely many times continuously differentiable, and we have
\begin{align} \label{eq:stabilityhonelin}
    \norm[H_0^1(\domain)]{y(\kappa,w)} \leq (2/\kappa_{\inf}^*)\norm[H^{-1}(\domain)]{w}
    \quad \text{for all} \quad
    (\kappa, w) \in \mathfrak{P}_{\text{ad}}.
\end{align}
Moreover, if~$u \in L^2(\domain)$,
then $y(\kappa,Bu) \in H^1_0(\domain) \cap H^2(\domain)$, and there exists a constant~$C_{H^2}>0$ such that
\begin{align} \label{eq:stabilityhtwolin}
	    \norm[H^2(\domain)]{y(\kappa,Bu)} \leq
	    \frac{C_{H^2}
	\norm[C^1(\bar{\domain})]{\kappa}^3
    }{(\kappa^*_{\inf}/2)^4} \norm[L^2(\domain)]{u}
     \quad \text{for all} \quad
    (\kappa, Bu) \in \mathfrak{P}_{\text{ad}},
     \quad
    u \in L^2(\domain).
\end{align}
Finally, the mapping
\begin{align*}
 \big\{\,\kappa  \in C^1(\bar{\domain})\;|\;\kappa > \kappa^*_{\inf}/2\,\big\} \times L^2(\domain) \ni   (\kappa,u) \mapsto y(\kappa,Bu) \in H^1_0(\domain)  \cap  H^2(\domain)
\end{align*}
is continuous.
\end{lemma}
In the following, let~$U_0 \subset L^2(\domain) $ be an arbitrary but fixed convex, bounded neighborhood of~$\adcsp$, and let~$C_{U_0}>0$ be such that~$\norm[L^2(\domain)]{u} \leq C_{U_0}$ for all $u \in U_0$.
Defining the parametrized control-to-state operator
\begin{align*}
    S \colon B(U_0) \times \Xi \to H^1_0(\domain), \quad \text{where}\quad  S(w,\xi)\coloneqq y(\kappa(\xi),w),
\end{align*}
we can verify the remaining assumptions in the linear case.

\begin{lemma} \label{lem:asslinear}
    If \Cref{assumption:poisson:objective}
    hold with $b = 0$, then
    $f $ in \eqref{eq:integrandpdeex} is Carath\'eodory, and
\begin{align*}
0 \leq \rpobj(Bu,\xi) \leq \varrho(C_{U_0}C^2_\domain (2/\kappa^*_{\inf}) )
\quad \text{for all} \quad (u,\xi) \in U_0 \times \Xi.
\end{align*}
\end{lemma}
\begin{proof}
    The Carath\'eodory property of~$f$ follows from the continuity of~$J$ and that of the mapping~$y$ from \Cref{lem:extlinpde}. Moreover, leveraging \Cref{itm:application_objective_lipschitz_gradient} and \eqref{eq:stabilityhonelin} yields
    \begin{align*}
        0 \leq \rpobj(Bu,\xi) \leq \varrho(C^2_\domain (2/\kappa^*_{\inf}) \norm[L^2(\domain)]{u}) \leq \varrho(C_{U_0}C^2_\domain (2/\kappa^*_{\inf})) \quad \text{for all} \quad (u,\xi) \in U_0 \times \Xi.
    \end{align*}
\end{proof}

Next, we verify~\Cref{assumption:gradient_consistency}.
\begin{lemma} \label{lem:ass2linear}
    Let \Cref{assumption:poisson:objective}
    hold with $b = 0$.
We define
\begin{align*}
\Du_w f \colon B(U_0) \times \Xi \to H^1_0(\domain), \quad \text{where} \quad\Du_w f(w, \xi) \coloneqq S\big(B \nabla J(B^* S(w, \xi)),\xi \big).
\end{align*}
The mapping~$\Du_w f$ is Carath\'eodory and, for every~$\xi \in \Xi$, the function
\begin{align*}
    g_\xi \colon U_0 \to \mathbb{R}, \quad \text{where} \quad  g_\xi(u) \coloneqq J(B^* S(Bu,\xi))
\end{align*}
is continuously differentiable with~$\nabla g_\xi(u)= B^*\Du_w f(Bu, \xi) $. Moreover, for all $ (u,\xi)\in U_0 \times \Xi$,
 \begin{align*}
    \norm[L^2(\domain)]{\nabla g_\xi (u)} \leq
             C_D \norm[H^1_0(\domain)]{\Du_w f(Bu,\xi)}
             \leq C^2_D (2/\kappa_{\inf}^*)  \sqrt{2\ell \varrho( C^2_D (2/\kappa_{\inf}^*) C_{U_0})}.
 \end{align*}
Finally, we have for all $u_1, u_2 \in \adcsp$ and $\xi \in \Xi$,
     \begin{align*}
	    \norm[L^2(\domain)]{\nabla g_\xi (u_1)-\nabla g_\xi (u_2)} \leq
	    C^4_D \ell (2/\kappa_{\inf}^*)^2 \norm[H^{-1}(\domain)]{Bu_1-Bu_2}.
	\end{align*}
\end{lemma}
\begin{proof}
The Carath\'eodory property of~$\Du_w f$ follows analogously to \Cref{lem:asslinear}.
Moreover, the differentiability of~$g_\xi$ and its gradient representation are  consequences of the chain rule as well as standard adjoint calculus, respectively. Fix $ (u,\xi)\in U_0 \times \Xi$. Using this representation and \eqref{eq:stabilityhonelin}, we obtain
\begin{align*}
    \norm[L^2(\domain)]{\nabla g_\xi (u)} \leq
             C_D \norm[H^1_0(\domain)]{\Du_w f(Bu,\xi)}
             \leq C^2_D (2/\kappa_{\inf}^*) \norm[L^2(\domain)]{\nabla J(B^* S(Bu, \xi))}.
\end{align*}
Furthermore, by applying~\eqref{eq:stability_estimate_gradient}, we have
\begin{align} \label{eq:estforgradaux}
    \norm[L^2(\domain)]{\nabla J(B^* S(Bu, \xi))}^2 \leq  2 \ell \varrho(\norm[L^2(\domain)]{B^*S(Bu, \xi)})
 \leq  2\ell \varrho( C^2_D (2/\kappa_{\inf}^*) C_{U_0}),
\end{align}
where the last inequality follows again by using~\eqref{eq:stabilityhonelin}. Combining both inequalities yields the desired estimate.
Fix $u_1, u_2 \in \adcsp$ and $\xi \in \Xi$. The Lipschitz estimate for~$\nabla g_\xi$ follows analogously noting that
 \begin{align*}
     \nabla g_\xi (u_1)-\nabla g_\xi (u_2)= B^* S\big(B \nabla J(B^* S(Bu_1, \xi))-B^* \nabla J(B^* S(Bu_2, \xi)),\xi \big),
 \end{align*}
 and thus
 \begin{align*}
     \norm[L^2(\domain)]{\nabla g_\xi (u_1)-\nabla g_\xi (u_2)} & \leq C^2_D (2/\kappa_{\inf}^*) \norm[L^2(\domain)]{\nabla J(B^* S(Bu_1, \xi))- \nabla J(B^* S(Bu_2, \xi))} \\
     & \leq C^3_D \ell (2/\kappa_{\inf}^*)  \norm[H^1_0(\domain)]{S(Bu_1, \xi)-S(Bu_2, \xi)} \\ & \leq C^4_D \ell (2/\kappa_{\inf}^*)^2 \norm[H^{-1}(\domain)]{Bu_1-Bu_2}.
 \end{align*}
\end{proof}

Next, we discuss sufficient conditions for \Cref{assumption:hoelder-growth-deviation} with the particular choice of~$\mathcal{U}=L^1(\domain)$ and~$\mathcal{H}=H^2(\domain)$. In this case, \Cref{assumption:hoelder} is satisfied. In the following, we denote by $u^*$ a solution to \eqref{eq:riskneutralpdeproblem}. We recall the identity $\Du F(Bu^*)= \E{\Du_w f(Bu^*,\boldsymbol{\xi})}$.
For a measurable set $D_0 \subset D$, we denote
by $\operatorname{meas}(D_0)$ its
Lebesgue measure.

\begin{proposition}
    \label{prob:linear:std}
    Let \Cref{assumption:poisson:objective}
    hold with $b = 0$.
    Then, the mapping~$\Xi \ni \xi \mapsto \Du_w f(Bu^*,\xi) \in H^2(\domain)$ is measurable, $\Du_w f ( Bu^*,\boldsymbol{\xi})$
is integrable, and
    \begin{align*}
    \E{\norm[H^2(\domain)]{\Du_w f ( Bu^*,\boldsymbol{\xi})-
        		\mathbb{E}[\Du_w f ( Bu^*,\boldsymbol{\xi})]}^2} \leq
           \frac{2\ell C^2_{H^2}
	\E{\norm[C^1(\bar{\domain})]{\kappa(\boldsymbol{\xi})}^6}
    }{(\kappa^*_{\inf}/2)^8} \varrho( C^2_D (2/\kappa_{\inf}^*) C_{U_0})
          <\infty.
    \end{align*}
    If, moreover, there exist constants ~$C>0$ and~$\nu \in (0,1]$ such that
    \begin{align*}
       \operatorname{meas}\big\{\,x \in \domain\; \colon \; \big||\lbrack \Du  F(Bu^*) \rbrack(x)|-\beta\big| \leq \varepsilon \,\big\} \leq C \varepsilon^\nu
       \quad \text{for all}
       \quad \varepsilon > 0,
    \end{align*}
    then there exists $\theta>0$ such that
  	\begin{align*}
	\inner[L^2(\domain)]{B^*\Du\erpobj(Bu^*)}{u-u^*}
	+ \psi(u) - \psi(u^*) \geq
	\theta \norm[L^1(\domain)]{u-u^*}^{1+1/\nu}
	\quad \text{for all} \quad u \in L^2(\domain).
	\end{align*}
\end{proposition}
\begin{proof}
    \Cref{lem:extlinpde} ensures $\Du_w f(Bu^*,\xi) \in H^2(\domain)$ as well as the claimed measurability.
    Using \eqref{eq:stabilityhtwolin} and \eqref{eq:estforgradaux}, we arrive at
    \begin{align*}
         \norm[H^2(\domain)]{\Du_w f ( Bu^*,\xi)}^2 &\leq
	    \frac{C^2_{H^2}
	\norm[C^1(\bar{\domain})]{\kappa(\xi)}^6
    }{(\kappa^*_{\inf}/2)^8} \norm[L^2(\domain)]{\nabla J(B^* S(Bu^*, \xi))}^2
\\
& \leq   \frac{2\ell C^2_{H^2}
	\norm[C^1(\bar{\domain})]{\kappa(\xi)}^6
    }{(\kappa^*_{\inf}/2)^8} \varrho( C^2_D (2/\kappa_{\inf}^*) C_{U_0}).
\end{align*}
Taking expectations yields the expectation bound.
Finally, the growth condition follows analogously from Proposition~5.5 in \cite{Kunisch2021}.
\end{proof}
Finally, we comment on the uniform Lipschitz continuity assumed in \Cref{coroll:fastrateresandgap}.
\begin{lemma}
    Let \Cref{assumption:poisson:objective}
    hold with $b = 0$.
  Suppose that  $\kappa \colon \Xi \to C^1(\bar \domain)$ from \Cref{assumption:poisson} satisfies $\|\kappa(\xi)\|_{C^1(\bar \domain)} \leq \kappa_{\sup}^* $ for all $\xi \in \Xi$
  and   some~$\kappa_{\sup}^*>0$. Then,
for all $u_1, u_2 \in \adcsp$ and $\xi \in \Xi$,
\begin{align*}
        \norm[L^\infty(\domain)]{\nabla g_\xi (u_1)-\nabla g_\xi (u_2)} \leq \ell \bigg( C_{H^2; L^\infty} \frac{C_{H^2}
	(\kappa_{\sup}^*)^3
    }{(\kappa^*_{\inf}/2)^4} \bigg)^2 \norm[L^1(\domain)]{u_1-u_2},
    \end{align*}
    where~$C_{H^2} > 0$ is the constant from \Cref{lem:extlinpde}, and~$C_{H^2; L^\infty} > 0$ denotes the embedding constant of~$H^2(\domain) \embedding L^\infty(\domain)$.
\end{lemma}
\begin{proof}
Similarly to the proof of \Cref{lem:ass2linear}, we obtain
    \begin{align*}
        \norm[L^\infty(\domain)]{\nabla g_\xi (u_1)-\nabla g_\xi (u_2)} & \leq \tfrac{C_{H^2; L^\infty} C_{H^2}
	(\kappa_{\sup}^*)^3
    }{(\kappa^*_{\inf}/2)^4} \norm[L^2(\domain)]{\nabla J(B^* S(Bu_1, \xi))- \nabla J(B^* S(Bu_2, \xi))} \\
     & \leq\tfrac{ \ell C_{H^2; L^\infty} C_{H^2}
	(\kappa_{\sup}^*)^3
    }{(\kappa^*_{\inf}/2)^4}  \norm[L^2(\domain)]{S(Bu_1, \xi)-S(Bu_2, \xi)}
    \end{align*}
    using~\eqref{eq:stabilityhtwolin} as well as the boundedness of~$\kappa$. Furthermore, by standard adjoint arguments and by again exploiting~\eqref{eq:stabilityhtwolin}, we arrive at
    \begin{align*}
        \norm[L^2(\domain)]{S(Bu_1, \xi)-S(Bu_2, \xi)}&=\sup_{\norm[L^2(\domain)]{v}\leq 1} \inner[L^2(\domain)]{S(Bu_1-Bu_2, \xi)}{v} \\ &  = \sup_{\norm[L^2(\domain)]{v}\leq 1} \inner[L^2(\domain)]{S(Bv, \xi)}{u_1-u_2} \\ & \leq C_{H^2; L^\infty} \sup_{\norm[L^2(\domain)]{v}\leq 1} \norm[H^2(\domain)]{S(Bv, \xi)} \norm[L^1(\domain)]{u_1-u_2} \\ & \leq
       C_{H^2; L^\infty} \frac{C_{H^2}
	(\kappa_{\sup}^*)^3
    }{(\kappa^*_{\inf}/2)^4} \norm[L^1(\domain)]{u_1-u_2}.
    \end{align*}
    Combining both observations yields the desired estimate.
\end{proof}
\subsection{Risk-neutral bilinear PDE-constrained optimization}
\label{subsect:bilinearproblem}
Next, we show that bilinear control problems, that is, $E(u,y)=uy$, also fit into our abstract setting. For this purpose, we assume $\lb(x) \geq 0$
for a.e.\ $x\in \domain$.
Similar to the previous example, consider the operator
\begin{align*}
    A \colon C^1(\bar{\domain})  \times L^2(\domain) \times H^1_0(\domain) \times L^2(\domain) \to H^{-1}(\domain),
\end{align*}
where
\begin{align*}
    \dualpHzeroone[\domain]{A(\kappa, b, y,u)}{v}\coloneqq\inner[L^2(\domain)^d]{\kappa \nabla y}{\nabla v}
+ \inner[L^2(\domain)]{y u}{v}
- \inner[L^2(\domain)]{b}{v}.
\end{align*}
The map $A$ is infinitely many times continuously differentiable.
Now, define the open set
\begin{align*}
\mathfrak{P}_{\text{ad}} \coloneqq  \big\{\,\kappa  \in C^1(\bar{\domain})\;|\;\kappa > \kappa^*_{\inf}/2\,\big\} \times L^2(\domain) \times \vadcsp ,
\end{align*}
where
\begin{align*}
\vadcsp \coloneqq \adcsp + (\kappa_{\inf}^*/(4C_{H_0^1;L^4}^2)) \cdot \mathbb{B}_{L^2(\domain)}(0).
\end{align*}
Here, $\mathbb{B}_{L^2(\domain)}(0)$
is the open $L^2(\domain)$-unit ball,
and $C_{H_0^1; L^4} > 0$ is the embedding constant of~$H^1_0(\domain) \embedding L^4(\domain)$.
We consider the equation
\begin{align} \label{eq:bilinearstrong}
\text{find}~y \in H^1_0(\domain) \colon \quad A(\kappa, b, y,u)=0, \quad \text{where} \quad (\kappa,b,u) \in \mathfrak{P}_{\text{ad}} .
\end{align}
The following result is a consequence of the implicit function theorem.
The constant $C_{H^2}$ in the following lemma is that from \Cref{lem:extlinpde}.
\begin{lemma} \label{lem:extpdegeneral}
Let \Cref{assumption:poisson:objective} hold with $\lb(x) \geq 0$
for a.e.\ $x\in \domain$.
For every~$(\kappa,b,u) \in \mathfrak{P}_{\text{ad}}$, there exists a unique~$y=y(\kappa,b,u) \in H^1_0 (\domain)$ satisfying~\eqref{eq:bilinearstrong}. The  mapping~$ \mathfrak{P}_{\text{ad}} \ni (\kappa,b,u) \mapsto y(\kappa,b,u) \in H^1_0(\domain)$
is infinitely many times continuously differentiable, and
\begin{align} \label{eq:stabilityhone}
    \norm[H_0^1(\domain)]{y(\kappa,b,u)} \leq C_\domain(4/\kappa_{\inf}^*)\norm[L^2(\domain)]{b}
    \quad \text{for all} \quad (\kappa,b,u) \in \mathfrak{P}_{\text{ad}}.
\end{align}
Moreover, if $u\in \adcsp$, then $y(\kappa,b,u) \in H^1_0(\domain) \cap H^2(\domain)$, we have
\begin{align} \label{eq:stabilityhtwo}
	    \norm[H^2(\domain)]{y(\kappa,b,u)} \leq
	    \frac{C_{H^2}
	\norm[C^1(\bar{\domain})]{\kappa}^3
    }{(\kappa^*_{\inf}/2)^4} \big(\norm[L^\infty(\domain)]{\ub} C^2_D
        +1 \big) \norm[L^2(\domain)]{b}.
\end{align}
and the mapping
\begin{align*}
 \big(\big\{\,\kappa  \in C^1(\bar{\domain})\;|\;\kappa > \kappa^*_{\inf}/2\,\big\} \times L^2(\domain) \times \adcsp \big)\ni   (\kappa,b,u) \mapsto y(\kappa,b,u) \in H^2(\domain)
\end{align*}
is continuous.
\end{lemma}
\begin{proof}
Given~$(\kappa,b,u) \in \mathfrak{P}_{\text{ad}}$, the existence of a unique solution to~\eqref{eq:bilinearstrong} and the a priori estimate in~\eqref{eq:stabilityhone} follow from the Lax--Milgram lemma and the definitions of $\mathfrak{P}_{\text{ad}}$ and~$\vadcsp$,  see also~\cite[sect.\ 7.1]{Milz2022d}. Similarly, the higher regularity and the estimate in~\eqref{eq:stabilityhtwo} follow by a standard bootstrapping argument. It remains to discuss the regularity of the solution mapping. Starting with the~$H^1_0(\domain)$-result, note that there holds
\begin{align*}
    A_y(\kappa, b, y,u) \delta y= A(\kappa, b, \delta y,0) \quad \text{for all} \quad \delta y \in H^1_0(\domain).
\end{align*}
Together with~\eqref{eq:stabilityhone}, this implies that~$A_y(\kappa, b, y,u)$ is a Banach space isomorphism. Hence, the smoothness of~$y(\kappa,u,b) \in H^1_0(\domain)$ with respect to its inputs follows from the implicit function theorem. In order to verify the~$H^2(\domain)$-continuity, let
\begin{align*}
   (\kappa_j, b_j,u_j) \in \big(\big\{\,\kappa  \in C^1(\bar{\domain})\;|\;\kappa > \kappa^*_{\inf}/2\,\big\} \times L^2(\domain) \times \adcsp \big), \quad j=1,2
\end{align*}
denote two admissible triples and let~$y_j \in H^2(\domain) \cap H^1_0(\domain)$ be the associated solutions of~\eqref{eq:bilinearstrong}. The difference~$\delta y= y_1-y_2 \in H^2(\domain) \cap H^1_0(\domain)  $ is the unique solution of
\begin{align*}
\text{find}~y \in H^1_0(\domain) \colon \quad A(\kappa_1, -A(\kappa_1, b_1, y_2,u_1), y,u_1)=0,
\end{align*}
where we identify
\begin{align*}
    A(\kappa_1, b_1, y_2,u_1)&=A(\kappa_1, b_1, y_2,u_1)-A(\kappa_2, b_2, y_2,u_2) \\&= -\nabla \cdot((\kappa_1-\kappa_2) \nabla y_2)+ (u_1-u_2) y_2- (b_1-b_2) \in L^2(\domain).
\end{align*}
Invoking~\eqref{eq:stabilityhtwo}, we thus get
\begin{align*}
    \norm[H^2(\domain)]{y_1-y_2} \leq
	    \frac{C_{H^2}
	\norm[C^1(\bar{\domain})]{\kappa_1}^3
    }{(\kappa^*_{\inf}/2)^4} \big(\norm[L^\infty(\domain)]{\ub} C^2_D
        +1 \big) \norm[L^2(\domain)]{ A(\kappa_1, b_1, y_2,u_1)},
\end{align*}
and together with
\begin{align*}
    \norm[L^2(\domain)]{ A(\kappa_1, b_1, y_2,u_1)} \leq C \big(\norm[C^1(\bar{\domain})]{\kappa_1-\kappa_2}+\norm[L^2(\domain)]{u_1-u_2}+\norm[L^2(\domain)]{b_1-b_2}\big),
\end{align*}
for some~$C=C(y_2)>0$, the claimed continuity result follows.
\end{proof}

From this point on, we would like to
argue as in the previous section by introducing an appropriate control-to-state operator $S$ in order to define the integrand~$f$. In the bilinear setting, however, this requires additional care, as the equation in~\eqref{eq:bilinearstrong} might not be well-posed for controls in a~$H^{-1}(\domain)$-neighborhood of~$B(\adcsp)$.
Taking this into account, we set~$U_0=\vadcsp$ and consider the mapping
\begin{align*}
    S \colon B(U_0) \times \Xi \to H_0^1(\domain),
    \quad \text{where} \quad S(w,\xi)= y(\kappa(\xi), b(\xi),  B^{-1}_r w).
\end{align*}
Here, $B_r \colon U_0 \to B(U_0)$ is the restriction of the compact operator $B$, and~$B^{-1}_r$ denotes its discontinuous inverse. Using the following lemma, we deduce the Lipschitz continuity of~$S$ with respect to the control.

\begin{lemma}
	\label{lem:multiplication}
 If $\domain \subset \mathbb{R}^d$ is a bounded domain, then there exists a constant
 $C_{H^2; H_0^1}>0$ such that for all
$v \in H_0^1(\domain)$ and $\varphi \in H^2(\domain)$, we have
     $v \varphi \in H_0^1(\domain)$
     and $\norm[H_0^1(\domain)]{v\varphi} \leq C_{H^2; H_0^1} \norm[H_0^1(\domain)]{v}\norm[H^2(\domain)]{\varphi}.
$
\end{lemma}
\begin{proof}
This is a consequence of
Theorem~7.4 in \cite{Behzadan2021}, the  $H^1(\domain)$-trace operator's continuity, and Friedrichs' inequality.
\end{proof}
\begin{lemma} \label{lem:contbilin}
If \Cref{assumption:poisson:objective} holds with $\lb(x) \geq 0$
for a.e.\ $x\in \domain$, $w_1,w_2 \in B(\adcsp)$, and $\xi \in \Xi$, then
\begin{align}
 \norm[H^1_0(\domain)]{S(w_1,\xi)-S(w_2,\xi)} \leq C_{H^2; H_0^1}  (4/\kappa_{\inf}^*)    \norm[H^{-1}(\domain)]{w_1 -w_2} \norm[H^2(\domain)]{S(w_1,\xi)}.
\end{align}
\end{lemma}
\begin{proof}
    Let~$u_1, u_2 \in\adcsp$ be such that~$w_j=Bu_j$,~$j=1,2$. Since~$b(\xi) \in L^2(\domain)$, we have~$S(Bu_1,\xi) \in H^2(\domain)$, see \Cref{lem:extpdegeneral}.
    Consequently, see \Cref{lem:multiplication},~$(u_2-u_1) S(Bu_1,\xi)$ induces a linear continuous functional on~$H^1_0(\domain)$ with
    \begin{align*}
        \dualpHzeroone[\domain]{(u_2-u_1) S(w_1,\xi)}{v} &= \int_\domain (u_2-u_1) S(B u_1,\xi) v ~\mathrm{d}x
        \\& \leq \norm[H^{-1}(\domain)]{w_1 -w_2} \norm[H_0^1(\domain)]{vS(B u_1,\xi)} \\
        & \leq C_{H^2; H_0^1} \norm[H^{-1}(\domain)]{w_1 -w_2} \norm[H^2(\domain)]{S(w_1,\xi)} \norm[H_0^1(\domain)]{v}
    \end{align*}
    for all~$v \in H^1_0(\domain)$. The claimed result then follows from~\eqref{eq:stabilityhone} noting that
    \begin{align*}
       S(w_1, \xi)-S(w_2, \xi)=y(\kappa(\xi),(u_2-u_1)S(w_1,\xi), u_2 ).
    \end{align*}
\end{proof}

The following lemma verifies \Cref{assumption:objective_consistency} for the integrand
$f$ in \eqref{eq:integrandpdeex}.
\begin{lemma} \label{lem:ass1bilinear}
Let \Cref{assumption:poisson:objective} hold with $\lb(x) \geq 0$
for a.e.\ $x\in \domain$.
    Then
    \begin{align*}
        0 \leq f(Bu, \xi) \leq \varrho\big( C^2_D (4/\kappa_{\inf}^*) \norm[L^2(\domain)]{b(\xi)} \big) \quad \text{for all} \quad (u,\xi) \in U_0 \times \Xi.
    \end{align*}
    Moreover, $f $ is continuous
        in its first argument
        on $B(\adcsp)$
        for each $\xi \in \Xi$
        and measurable in its second
        argument on $\Xi$ for each $u \in B(\csp_0)$.
\end{lemma}
\begin{proof}
    The measurability and continuity statements follow from the continuity of~$J$
    and that of $(\kappa,b,u) \mapsto y(\kappa,b,u)$ (see \Cref{lem:extpdegeneral}), \Cref{itm:application_objective_lipschitz_gradient},  and the Lipschitz continuity of~$S$ for fixed~$\xi \in \Xi$. The stability estimate follows similarly to \Cref{lem:asslinear} by using~\eqref{eq:stabilityhone}.
\end{proof}

Next, we verify \Cref{assumption:gradient_consistency,assumption:Lipschitz_continuity}. We define
\begin{align*}
    C_{\mathrm{sq}} \coloneqq
    \sqrt{2 \ell \varrho\big( C^2_D (4/\kappa_{\inf}^*) b_{\sup}^*\big)}.
\end{align*}
\begin{lemma} \label{lem:ass2bilinear}
Let \Cref{assumption:poisson:objective} hold with $\lb(x) \geq 0$
for a.e.\ $x\in \domain$.
For each $\xi \in \Xi$, consider the mapping
\begin{align*}
g_\xi \colon U_0 \to \mathbb{R}, \quad \text{where} \quad     g_\xi (u) \coloneqq f(Bu,\xi)= J( B^* y(\kappa(\xi), b(\xi),u)).
\end{align*}
The function~$g_\xi$ is continuously differentiable on~$U_0$, and there holds
 \begin{align*}
     \nabla g_\xi (u) =-p_{u,\xi} S(Bu,\xi) \in L^2(\domain), \quad \text{where} \quad p_{u,\xi}\coloneqq y(\kappa(\xi),\nabla J(B^* S(Bu,\xi)), u).
 \end{align*}
 Moreover, for every $(u,\xi) \in \adcsp \times \Xi$, we have~$p_{u,\xi} S(Bu,\xi) \in H^1_0(\domain)$, the mapping
\begin{align*}
\Du_w f \colon B(\adcsp) \times \Xi \to H^1_0(\domain), \quad \text{where} \quad\Du_w f(w, \xi)\coloneqq-p_{u,\xi} S(w,\xi),
\end{align*}
is Carath\'eodory, and we have the a priori estimates
\begin{align*}
             \norm[L^2(\domain)]{\nabla g_\xi (u)} \leq  C^2_{H^1_0; L^4} C^2_D (4/\kappa_{\inf}^*)^2  C_{\mathrm{sq}}\norm[L^2(\domain)]{b(\xi)} \quad \text{for all} \quad (u,\xi)\in U_0 \times \Xi,
 \end{align*}
as well as for all $(u,\xi) \in \adcsp \times \Xi$,
    \begin{align*}
        \norm[H^1_0(\domain)] {\Du_w f(Bu,\xi)} & \leq C_{H^2; H_0^1}C_D (4/\kappa_{\inf}^*) \frac{C_{H^2}
	\norm[C^1(\bar{\domain})]{\kappa(\xi)}^3
    }{(\kappa^*_{\inf}/2)^4} \big(\norm[L^\infty(\domain)]{\ub} C^2_D
        +1 \big)   C_{\mathrm{sq}}\norm[L^2(\domain)]{b(\xi)}.
    \end{align*}
    Moreover, there exists an integrable random variable~$L_{\Du \rpobj} \colon \Xi \to [0,\infty)$ such that
    \begin{align*}
          \norm[L^2(\domain)]{\nabla g_\xi (u_1)-\nabla g_\xi (u_2)}  \leq L_{\Du \rpobj}(\xi)  \norm[H^{-1}(\domain)]{Bu_1 -Bu_2} \quad \text{for all} \quad u_1,u_2 \in \adcsp,
          \quad \xi \in \Xi.
    \end{align*}
\end{lemma}
\begin{proof}
The statement on the differentiability of~$g_\xi$ is a consequence of the chain rule  and  adjoint calculus. If~$u \in \adcsp $, then \Cref{lem:extpdegeneral} implies~$S(Bu, \xi) \in H^2(\domain) \cap H^1_0(\domain)$. Hence, \Cref{lem:multiplication} ensures $p_{u,\xi} S(Bu, \xi) \in H^1_0(\domain)$. The Carath\'eodory property of~$\Du_w f$ follows from the continuity of~$y$ and~$\nabla J$.
Invoking the stability estimates of \Cref{lem:extpdegeneral}, we further obtain
    \begin{align*}
        \norm[H^1_0(\domain)]{p_{u,\xi}} &\leq  C_D (4/\kappa_{\inf}^*)\norm[L^2(\domain)]{\nabla J(B^* S(Bu,\xi))} \leq  C_D (4/\kappa_{\inf}^*)  \sqrt{2 \ell J(B^* S(Bu,\xi)))}
    \end{align*}
    for all~$(u,\xi)\in U_0 \times \Xi$, and
    \begin{align} \label{eq:hzwop}
        \norm[H^2(\domain)]{p_{u,\xi}} \leq  \frac{C_{H^2}
	\norm[C^1(\bar{\domain})]{\kappa(\xi)}^3
    }{(\kappa^*_{\inf}/2)^4} \big(\norm[L^\infty(\domain)]{\ub} C^2_D
        +1 \big) C_{\mathrm{sq}}
        \quad \text{for all}
        \quad (u,\xi) \in \adcsp \times \Xi.
    \end{align}
    As a consequence, we obtain
    \begin{align*}
        \norm[L^2(\domain)]{\nabla g_\xi (u)} &\leq  \norm[L^4(\domain)]{p_{u,\xi}} \norm[L^4(\domain)]{S(Bu,\xi)} \leq C^2_{H^1_0; L^4} \norm[H^1_0(\domain)]{p_{u,\xi}} \norm[H^1_0(\domain)]{S(Bu,\xi)}
\\&\leq
        C^2_{H^1_0; L^4} C^2_D (4/\kappa_{\inf}^*)^2  C_{\mathrm{sq}}\norm[L^2(\domain)]{b(\xi)}
    \end{align*}
    for all~$(u,\xi) \in  U_0 \times \Xi$. Moreover, if $(u,\xi) \in \adcsp \times \Xi$, then
    \begin{align*}
        &\norm[H^1_0(\domain)]{\Du_w f(Bu,\xi)} \leq C_{H^2; H_0^1}\norm[H^2(\domain)]{S(Bu,\xi)} \norm[H^1_0(\domain)]{p_u}
        \\ &\quad \leq C_{H^2; H_0^1}C_D (4/\kappa_{\inf}^*) \frac{C_{H^2}
	\norm[C^1(\bar{\domain})]{\kappa(\xi)}^3
    }{(\kappa^*_{\inf}/2)^4} \big(\norm[L^\infty(\domain)]{\ub} C^2_D
        +1 \big)   C_{\mathrm{sq}}\norm[L^2(\domain)]{b(\xi)}.
    \end{align*}
    In particular, owing to \Cref{assumption:poisson:objective}, \eqref{eq:stabilityhtwo} and~\eqref{eq:hzwop}, there exists a square integrable random variable~$\zeta_{\text{max}} \colon \Xi \to [0,\infty)$ such that
    \begin{align*}
    \norm[L^2(\domain)]{\nabla g_\xi(u)} \leq \zeta_{\text{max}}(\xi) \quad \text{for all} \quad (u,\xi)\in U_0 \times \Xi,
    \end{align*}
    as well as
    \begin{align*}
        \max \big\{\norm[H^1_0(\domain)]{\Du_w f(Bu,\xi)}, \norm[H^2(\domain)]{S(Bu,\xi)}, \norm[H^2(\domain)]{p_{u,\xi}} \big\} \leq \zeta_{\text{max}}(\xi) \;\text{for all}
        \; (u,\xi) \in \adcsp \times \Xi.
    \end{align*}

    To establish the claimed Lipschitz continuity, fix~$u_1, u_2 \in \adcsp$ and $\xi \in \Xi$, and split
    \begin{align*}
        p_{u_1, \xi}-p_{u_2,\xi}= p_{u_1, \xi}-y(\kappa(\xi),\nabla J(B^* S(Bu_2,\xi)), u_1)+y(\kappa(\xi),\nabla J(B^* S(Bu_2,\xi)), u_1)-p_{u_2,\xi}.
    \end{align*}
    For estimating the norm of the first difference, we note that
    \begin{align*}
        \norm[L^2(D)]{\nabla J(B^* S(Bu_1,\xi))-\nabla J(B^* S(Bu_2,\xi))} &\leq \ell C_D \norm[H^1_0(\domain)]{S(Bu_1,\xi)-S(Bu_2,\xi)} \\ &\leq \frac{4\ell C_D C_{H^2; H_0^1} \zeta_{\text{max}}(\xi)}{\kappa_{\inf}^*} \norm[H^{-1}(\domain)]{Bu_1 -Bu_2}
    \end{align*}
    using \Cref{lem:contbilin}. Similarly, arguing along the lines of proof in \Cref{lem:contbilin}, we get
    \begin{align*}
        \norm[H^1_0(\domain)]{y(\kappa(\xi),\nabla J(B^* S(Bu_2,\xi)), u_1)-p_{u_2,\xi}} \leq C_{H^2; H_0^1}  (4/\kappa_{\inf}^*)    \norm[H^{-1}(\domain)]{Bu_1 -Bu_2} \zeta_{\text{max}}(\xi).
    \end{align*}
    As a consequence, we have
    \begin{align*}
        \norm[H^1_0(D)]{p_{u_1}-p_{u_2}} \leq  C_{H^2; H_0^1}  (4/\kappa_{\inf}^*) \zeta_{\text{max}}(\xi)\big(     1+\ell C^2_D  (4/\kappa_{\inf}^*)   \big)  \norm[H^{-1}(\domain)]{Bu_1 -Bu_2}.
    \end{align*}
    Finally, we decompose~$\nabla g_\xi (u_1)-\nabla g_\xi (u_2)$ and estimate
    \begin{align*}
        \norm[H^1_0(\domain)]{\nabla g_\xi (u_1)-\nabla g_\xi (u_2)}  &\leq \big(\norm[H^1_0(\domain)]{p_{u_1}-p_{u_2}}  + \norm[H^1_0(\domain)]{S(Bu_1,\xi)-S(Bu_2,\xi)} \big) \zeta_{\text{max}}(\xi)
        \\ & \leq C \zeta_{\text{max}}(\xi)^2  \norm[H^{-1}(\domain)]{Bu_1 -Bu_2}
    \end{align*}
    for some~$C>0$. By Friedrichs' inequality, the claimed Lipschitz statement follows.
\end{proof}

\Cref{lem:ass2bilinear}
verifies \Cref{assumption:gradient_consistency,assumption:Lipschitz_continuity}.
Moreover, as in the linear setting, \Cref{assumption:subgaussian_gradients} is satisfied since~$\nabla g_\xi$ is uniformly bounded on~$\adcsp \times \Xi$.

\section{Numerical illustrations}
\label{sec:numerics}
We present numerical illustrations for two instances of the risk-neutral problems analyzed
in \Cref{sec:bangbangcontrol}. The section's main objective is to illustrate
the theoretical results established in \Cref{subsect:sample_size_estimates,subsec:expectationboundsconvexproblems}.
Before presenting numerical results, we discuss problem data,
discretization aspects, implementation details, and the computation of reference solutions
and gap functions.

\paragraph{Problem data}
We chose  \(\domain = (0,1)^2\),
$b(x) \coloneqq 10 \sin(2 \pi  x_1-x_2) \cos(2 \pi x_2)$ in \eqref{eq:bilinear}, and
$J(v) \coloneqq (1/2) \|v-y_d\|_{L^2(\domain)}^2$
in \eqref{eq:riskneutralpdeproblem} with
\(y_d(x) \coloneqq \sin(2 \pi x_1) \sin(2 \pi x_2) \exp(2 x_1)/6\). For the random diffusion coefficient \(\kappa\), we used a truncated Karhunen--Lo\`eve expansion with the eigenpairs described in Example~7.56 of \cite{Lord2014}. Our truncated expansion includes $100$ addends, uses a correlation length of $1$,
and independent truncated normal random variables, each without shift, a standard deviation \(1\), and a truncation interval of \([-3, 3]\). The remaining problem data are problem-dependent and introduced in \Cref{subsubsect:Linear,subsubsect:bilinear}.

\paragraph{Reference solutions and gap functions}
As the risk-neutral problems in \Cref{subsect:bangbanglinear} may lack closed-form solutions,
we compute reference solutions. We constructed $N_{\text{ref}} \coloneqq 8192$
samples
using a scrambled Sobol' sequence,
generated using \texttt{SciPy}'s quasi-Monte Carlo engine \cite{Roy2023}.
These samples are transformed via the inverse transformation method
and their subsequent values referred to as reference samples.
The authors chose to use a scrambled Sobol' sequence instead of Monte Carlo samples to generate the reference samples. This choice was not guided by theoretical error estimates such as those in Theorem~1 of \cite{Owen1998}, but rather by empirical observations made by one of the authors.
We used the reference samples to define a reference problem.
This reference problem is an SAA problem defined by the reference samples
with $N = N_{\text{ref}}$. We refer to its solutions/critical points
as reference SAA solutions/critical points.
Moreover, we approximate the gap functional $\Psi$
in \eqref{def:gapfunctional} using
the gap functional $\Psi_{\text{ref}}$ corresponding to
the reference SAA problem. Similarly, we approximate the objective function
$\objective$ via $\objective_{\text{ref}}$.

\paragraph{Discretization and implementation details}
To obtain finite-dimensional optimization problems, we discretized the state space \(H_0^1(\domain)\) using piecewise linear continuous finite element functions and \(L^2(\domain)\) using piecewise constant functions. These functions are defined on a regular triangulation of \(\domain\) with  \(n \coloneqq 64\) cells in each direction.
We refer to
\begin{align*}
    \min_{u \in \csp}\, f(Bu, \E{\boldsymbol{\xi}}) + \psi(u)
\end{align*}
as the nominal problem to \eqref{eq:nonconvexriskneutral}.
We solved the discretized problems using a conditional gradient method \cite{Kunisch2021}. Our implementation  is archived at \cite{Milz2024a}. The method stops when the gap functional falls below \(10^{-10}\) or after 100 iterations, whichever comes first.
Our simulation environment is built on \texttt{dolfin-adjoint} \cite{Funke2013,Mitusch2019}, \texttt{FEniCS} \cite{Alnes2015,Logg2012}, and \texttt{Moola} \cite{Nordaas2016}. The numerical simulations, except those for the nominal problems, were performed on the PACE Phoenix cluster \cite{PACE2017}.
The code and simulation output are archived at \cite{johannes_milz_2024_13145219}.

For each problem, we display both the nominal solutions and the SAA reference solutions side by side. These graphical illustrations indicate that nominal and stochastic solutions may exhibit different characteristics.
This presentation is not intended as a direct  
comparison between nominal and SAA reference solutions.

\subsection{Affine-linear problem}
\label{subsubsect:Linear}

We chose  $\lb = -1$, $\ub = 1$, and
$\beta = 0.0075$.
\Cref{fig:convex:nominal_reference_solutions} depicts
a nominal solution and a reference SAA solution.
\Cref{fig:convex:convergence_rates_objective_distance} depicts convergence
rates for empirical estimates of the optimality gap $\E{\objective_{\text{ref}}(u_N^*)-\objective_{\text{ref}}(u^*)}$ and
$\E{\|u_{N}^*-u^*\|_{L^1(\domain)}}$ over the sample size $N$.
We used $40$ realizations to estimate these means, and computed
convergence rates using least squares.
\Cref{fig:convex:convergence_rates_gap} depicts convergence
rates of the SAA gap function's empirical means.
The empirical convergence rates closely match the theoretical ones.

\begin{figure}[t!]
	\centering
		\subfloat{%
		\includegraphics[width=0.400\textwidth]
		{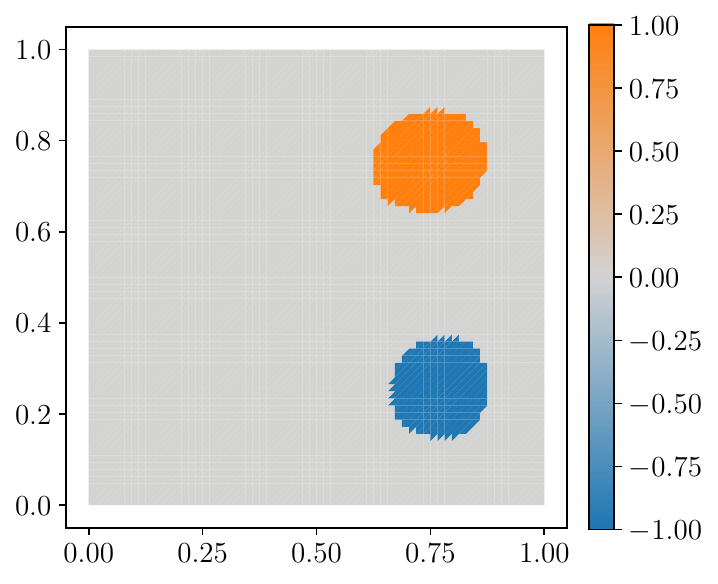}
        }
	\subfloat{%
		\includegraphics[width=0.400\textwidth]
		{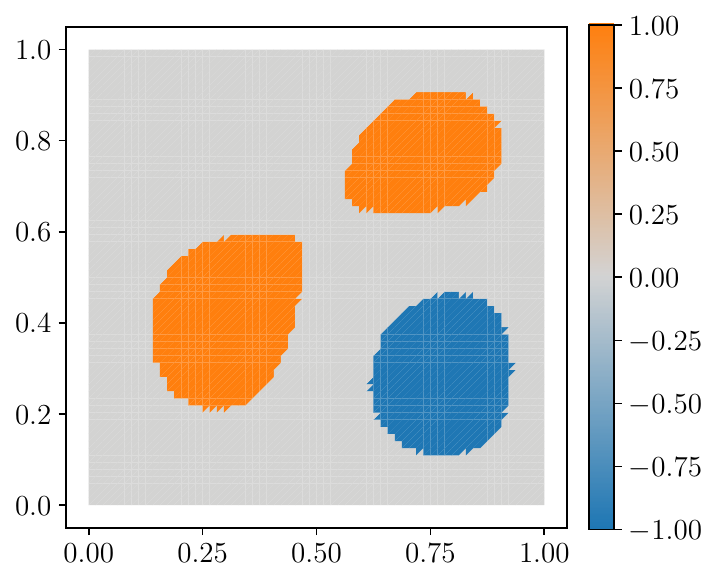}
        }
	\caption{
	For the affine-linear control
    problem,
	nominal solution \emph{(left)}
    and reference SAA solution $u^*$
    with $N = N_{\text{ref}}$ \emph{(right)}.
    }
	\label{fig:convex:nominal_reference_solutions}
\end{figure}

\begin{figure}[t!]
	\centering
	\subfloat{%
		\includegraphics[width=0.400\textwidth]
		{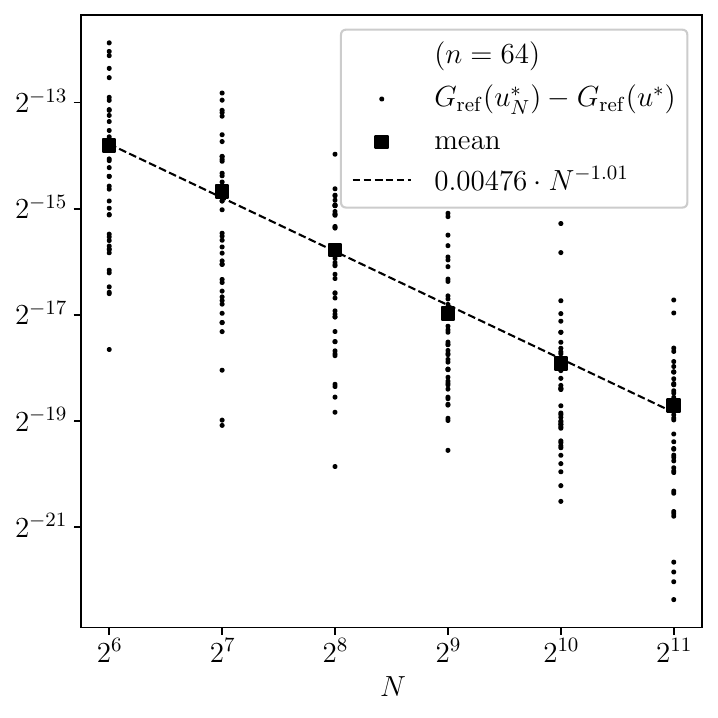}
        }
	\subfloat{%
		\includegraphics[width=0.400\textwidth]
		{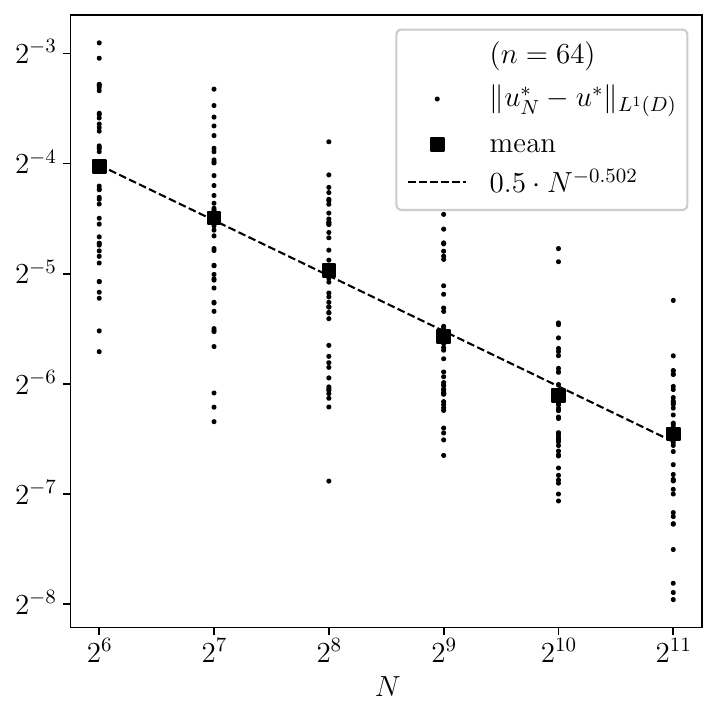}
        }
	\caption{%
		For the affine-linear control
    problem,
	empirical estimate of
    $\E{\objective_{\text{ref}}(u_N^*)-\objective_{\text{ref}}(u^*)}$
    over $N$
    \emph{(left)} and empirical estimate of
    $\E{\norm[L^1(\domain)]{u_{N}^*-u^*}}$
     over $N$
    \emph{(right)}.}
	\label{fig:convex:convergence_rates_objective_distance}
\end{figure}

\begin{figure}[t]
	\centering
	\subfloat{%
		\includegraphics[width=0.400\textwidth]
		{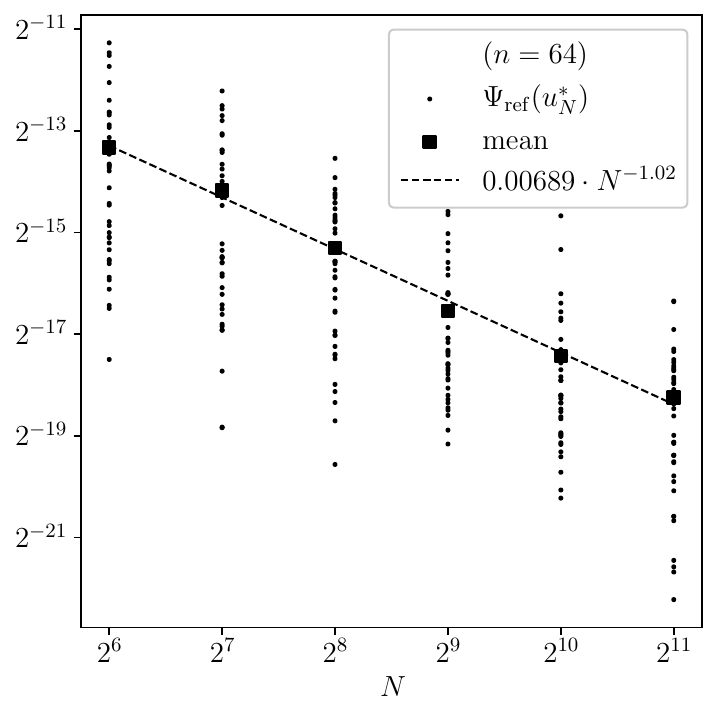}
    }
    \subfloat{%
    	\includegraphics[width=0.400\textwidth]
    	{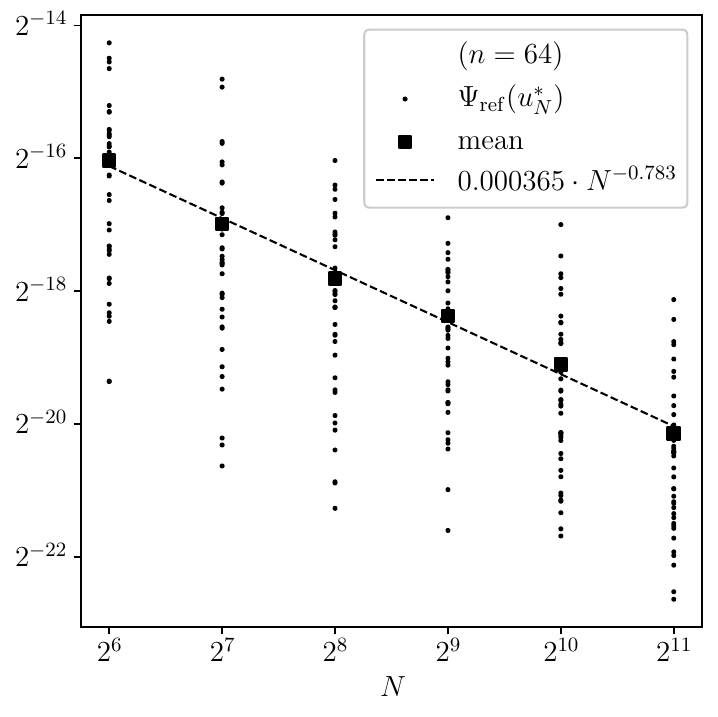}
    }
	\caption{%
  	Empirical estimate of
    $\E{\Psi_{\text{ref}}(u_{N}^*)}$
    over $N$
    for the affine-linear control problem \emph{(left)}
    and bilinear control problem \emph{(right)}.}
	\label{fig:convex:convergence_rates_gap}
\end{figure}

\subsection{Bilinear problem}
\label{subsubsect:bilinear}

We chose the constant control bounds $\lb = 0$ and $\ub = 1$, and $\beta = 0.00055$.
\Cref{fig:nonconvex:nominal_reference_solutions} depicts
nominal critical points and a reference SAA critical point.
\Cref{fig:convex:convergence_rates_gap} depicts convergence
rates of the SAA gap function's empirical means. These rates are faster than predicted by the theory, see \Cref{cor:sample_size_estimates_gap_function}. We think that this can be attributed to the fact that the covering number approach does not exploit higher-order regularity of the integrand and, in particular, potential local curvature around isolated minimizers. While a closer inspection of this improved convergence behavior is certainly of interest, it goes beyond the scope of the current work and is left for future research.

\begin{figure}[t!]
	\centering
	\subfloat{%
		\includegraphics[width=0.400\textwidth]
		{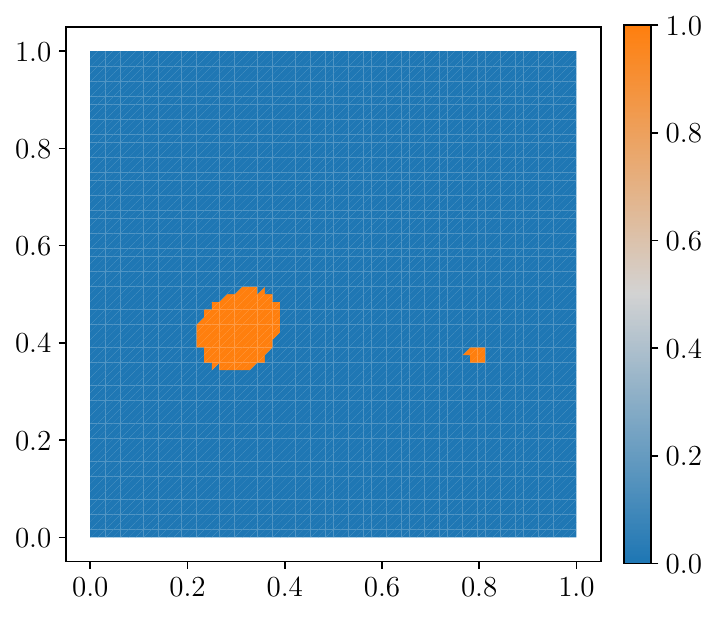}
    }
	\subfloat{%
		\includegraphics[width=0.400\textwidth]
		{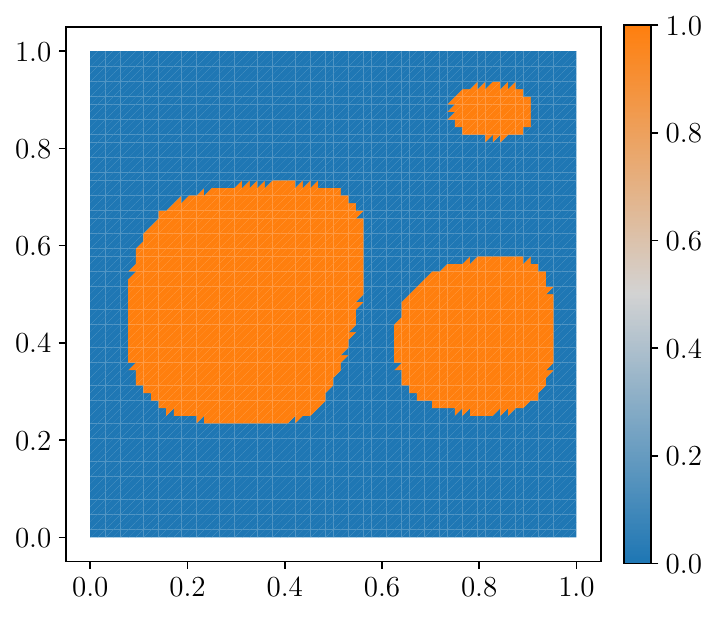}
    }
	\caption{%
    For the bilinear control
    problem,
	nominal critical point \emph{(left)}
    and reference SAA critical point $u^*$
    with $N = N_{\text{ref}}$ \emph{(right)}.
    }
	\label{fig:nonconvex:nominal_reference_solutions}
\end{figure}

\section{Discussion}
We have analyzed the SAA approach for risk-neutral optimization problems with nonsmooth but convex regularization terms. Our main results address the asymptotic consistency of this scheme, and the derivation of nonasymptotic sample size estimates for various optimality measures. The latter, as well as the employed techniques, come in two different flavors: for the general nonconvex case, we provide estimates on the expected value of the gap functional by applying the covering number approach. For convex objectives, and by relying on common growth conditions, we prove stronger results including convergence rates for the expected distance between minimizers,  improved estimates for the gap functional, and the suboptimality in the objective function value. The presented abstract framework is applied to both linear and bilinear PDE-constrained problems under uncertainty. We also use these applications to empirically demonstrate the sharpness of our convergence guarantees for the convex case.

Our investigation also raises new questions about the SAA approach for nonsmooth minimization. These include, for example, the extension of the results in \Cref{subsec:expectationboundsconvexproblems} to nonconvex problems assuming suitable second-order optimality conditions and taking into account the potential existence of multiple global and/or local minimizers. Moreover, an extension of our  results to risk-averse stochastic optimization and variational inequalities may be of interest.

\subsection*{Reproducibility of computational results}
Computer code that allows the reader to reproduce the computational results in
this manuscript is available at
\url{https://doi.org/10.5281/zenodo.13336970}.

\section*{Acknowledgments}
We thank the anonymous referees for their valuable comments, which have 
contributed to the improvement of our manuscript.

\appendix

\section{Uniform expectation bounds for expectation mappings}

We establish essentially known uniform expectation bounds for integrands defined on potentially infinite-dimensional
spaces.
While the techniques used in this section are
standard, the bounds are instrumental for
establishing one of our main results in \Cref{subsect:sample_size_estimates}.

Let $(\Theta, \Sigma,\mu)$ be a complete probability space,
and let $\boldsymbol{\zeta}$ be a random element with image space $\Theta$.
Moreover, let $\boldsymbol{\zeta}^1, \boldsymbol{\zeta}^2, \ldots$ be a sequence of independent $\Theta$-valued random elements
defined on a complete probability space, each $\boldsymbol{\zeta}^i$ having the same distribution as $\boldsymbol{\zeta}$.

\begin{assumption}[{sub-Gaussian \Caratheodory\ function, compact linear operator}]
    \label{assumption:uniformexpectationbounds}
    ~
    \begin{enumthm}[wide,nosep,leftmargin=*]
        \item The set $V_0$ is a nonempty, closed,
        bounded, convex
        subset of a reflexive Banach space
        $V$, $W$ is a Banach space, and $B : V \to W$ is linear and compact.
        \item The space $H$ is a separable Hilbert space,
        and $\mathsf{G} : B(V_0) \times \Theta \to H$
        is \Caratheodory.
        \item For each $w \in B(V_0)$,
        $\mathsf{G}(w,\cdot)$ is integrable, and for an integrable random variable $L_{\mathsf{G}} : \Theta \to [0,\infty)$,
        $\mathsf{G}(\cdot,\zeta)$ is Lipschitz continuous
        with Lipschitz constant $L_{\mathsf{G}}(\zeta)$
        for each $\zeta \in \Theta$.
        \item There exists $\tau_{\mathsf{G}}> 0$ such that for all $w \in B(V_0)$,
        $
            \E{\exp(\tau_{\mathsf{G}}^{-2}\norm[H]{\mathsf{G}(w,\boldsymbol{\zeta}) - \E{\mathsf{G}(w,\boldsymbol{\zeta})}}^2)} \leq \mathrm{e}
        $.
    \end{enumthm}
\end{assumption}

We define $g(w) \coloneqq \E{\mathsf{G}(w,\boldsymbol{\zeta})}$,
$\hat{g}_N(w) \coloneqq (1/N) \sum_{i=1}^N \mathsf{G}(w,\boldsymbol{\zeta}^i)$, and $L \coloneqq \E{L_{\mathsf{G}}(\boldsymbol{\zeta})}$.

\begin{proposition}
    \label{prop:uniformexpectationbounds}
    If \Cref{assumption:uniformexpectationbounds} holds, then
    for each $\nu > 0$ and  $N \in \mathbb{N}$,
    \begin{align*}
        \E{\sup_{v\in V_0} \norm[H]{g(Bv) - \hat{g}_N(Bv)}}
        &=
        \E{\sup_{w\in B(V_0)} \norm[H]{g(w) - \hat{g}_N(w)}}
        \\
        & \leq 2L\nu + \tfrac{\sqrt{3}\tau_{\mathsf{G}}}{\sqrt{N}} \big(\ln(2\mathcal{N}(\nu; B(V_0)))\big)^{1/2}.
    \end{align*}
\end{proposition}

\begin{proof}
    The proof is inspired by those
	of Theorems~9.84 and 9.86 in \cite{Shapiro2021}.
	We have
     $
        \norm[H]{g(w_2)-g(w_1)} \leq
        L \norm[W]{w_2-w_1}
    $
    and
    $
        \norm[H]{\hat{g}_N(w_2)-\hat{g}_N(w_1)} \leq
        \hat{L}_N \norm[W]{w_2-w_1}
    $
    for all $w_1, w_2 \in B(V_0)$,
    where  $\hat{L}_N \coloneqq (1/N) \sum_{i=1}^N L_{\mathsf{G}}(\boldsymbol{\zeta}^i)$.
    We define $K \coloneqq  \mathcal{N}(\nu; B(V_0))$.
    Since $B(V_0)$ is compact,
    $K$ is finite and there exist $w_1, \ldots, w_K \in B(V_0)$
    such that for each $w \in B(V_0)$,
    we have $\norm[W]{w-w_{k(w)}} \leq \nu$,
    where $k(w) \coloneqq  \mathrm{\arg\, \min}_{1\leq k\leq K}\, \norm[W]{w-w_k}$.
    For all $w \in B(V_0)$,
	\begin{align*}
	\begin{aligned}
	\norm[H]{\hat{g}_N(w)-g(w)}
	&\leq 	\norm[H]{\hat{g}_N(w)-\hat{g}_N(w_{k(w)})}
	+ \norm[H]{\hat{g}_N(w_{k(w)})-g(w_{k(w)})}  \\ & \quad + \norm[H]{g(w_{k(w)})-g(w)}
	\\
	& \leq \hat{L}_N \nu
	+ \norm[H]{\hat{g}_N(w_{k(w)})-g(w_{k(w)})}
	+ L \nu.
	\end{aligned}
	\end{align*}
    Hence
    \begin{align*}
        \sup_{v\in V_0} \norm[H]{g(Bv) - \hat{g}_N(Bv)}
        &= \sup_{w \in B(V_0)}
        \norm[H]{\hat{g}_N(w)-g(w)}
        \\
        &\leq \hat{L}_N \nu
	+ \max_{1\leq k\leq K}\, \norm[H]{\hat{g}_N(w_{k})-g(w_{k})}
	+ L \nu.
    \end{align*}
    Using Theorem~3 in \cite{Pinelis1986}
    and Lemma~1 in \cite{Milz2021}, we have
    $$\E{\cosh(\lambda\norm[H]{\hat{g}_N(w_{k})-g(w_{k})})} \leq
    \exp(3\lambda^2\tau_{\mathsf{G}}^2/(4N))
    \quad \text{for all}
    \quad \lambda \in \mathbb{R},
    \quad k = 1, \ldots, K.
    $$
    Now, Lemma~B.5 in \cite{Milz2022b} ensures
    $\E{\max_{1\leq k\leq K}\, \norm[H]{\hat{g}_N(w_{k})-g(w_{k})}} \leq
    \sqrt{3}\tau_{\mathsf{G}}\sqrt{\ln(2K)/N}$.
    Combined with $\E{\hat{L}_N} = L$, we obtain the expectation bound.
\end{proof}

\begin{footnotesize}

\end{footnotesize}
\end{document}